\documentclass[a4paper,reqno]{amsart}%
\usepackage{fullpage}
\usepackage{amsfonts}
\usepackage{amsmath}
\usepackage{amssymb}
\usepackage{graphicx}%
\providecommand{\U}[1]{\protect\rule{.1in}{.1in}}
\newtheorem{theorem}{Theorem}
\theoremstyle{plain}

\newtheorem{corollary}[theorem]{Corollary}

\newtheorem{lemma}[theorem]{Lemma}

\newtheorem{proposition}[theorem]{Proposition}

\numberwithin{theorem}{section}
\numberwithin{equation}{section}
\begin{document}
\title[From Lomonosov Lemma...]{From Lomonosov Lemma to Radical Approach in Joint Spectral Radius Theory}
\author{Edward Kissin}
\address{E. Kissin: STORM, London Metropolitan University, 166-220 Holloway Road,
London N7 8DB, Great Britain; }
\email{e.kissin@londonmet.ac.uk}
\author{Victor S. Shulman}
\address{V. S. Shulman: Department of Mathematics, Vologda State University, Vologda,
Russia; }
\email{shulman.victor80@gmail.com}
\author{Yurii V. Turovskii}
\address{Yu. V. Turovskii}
\email{yuri.turovskii@gmail.com}
\thanks{This paper is in final form and no version of it will be submitted for
publication elsewhere.}
\dedicatory{To the memory of Victor Lomonosov, a man who moved mountains in Mathematics}
\begin{abstract}
In this paper we discuss the infinite-dimensional generalizations of the
famous theorem of Berger-Wang (generalized Berger-Wang formulas) and give an
operator-theoretic proof of I. Morris's theorem about coincidence of three
essential joint spectral radius, related to these formulas. Further we develop
Banach-algebraic approach based on the theory of topological radicals, and
obtain some new results about these radicals.

\end{abstract}
\maketitle

\section{Introduction}

\subsection{Banach-algebraic consequences of Lomonosov Lemma}

The famous Lomonosov Lemma \cite{L73} states:\medskip

\textit{If an algebra }$A$\textit{ of operators on a Banach space }$X$\textit{
contains a non-zero compact operator }$T$\textit{ then either }$A$\textit{ has
a non-trivial closed invariant subspace }(IS, for brevity)\textit{ or it
contains a compact operator with a non-zero point in spectrum.\medskip}

An immediate consequence of this result is that \textit{any algebra of compact
quasinilpotent operators has an IS}; the standard technique gives then that
such an algebra is triangularizable.

M. G. Krein proposed to call compact quasinilpotent operators \textit{Volterra
operators}; respectively, a set of operators is called \textit{Volterra }if
all its elements are Volterra operators. Thus \textit{any Volterra algebra has
an IS}. This result was extended by the second \textbf{ }author \cite{Sh84} as
follows:\textit{\medskip}

\textit{Any Volterra algebra }$A$\textit{\textit{ has an IS which is also
invariant for all operators commuting with }}$A$ (such subspaces are
called\textit{ \textit{hyperinvariant}})\textit{.\medskip\ }

Besides of\textbf{ }Lomonosov's technique the proof used estimations of the
norms of products for elements of a Volterra algebra $A$; in fact, it was
proved in \cite{Sh84} that\textit{\textit{ the joint spectral radius}}
$\rho(M)$\textit{\textit{ of any finite set }}$M\subseteq A$\textit{\textit{
equals} }$0$, i.e. $A$ is \textit{finitely quasinilpotent}.

This result can be considered as an application of the invariant subspace
theory to the theory of joint spectral radius. Conversely, the second part of
the proof in \cite{Sh84} is an application of the joint spectral radius
technique to the invariant subspace theory (again via Lomonosov's theorem
about Volterra algebras): if $M=\{T_{1},...,T_{n}\}$ and $\rho(M)=0$ then
\[
\rho\left(  \sum_{i=1}^{n}T_{i}S_{i}\right)  =0
\]
for all operators $S_{i}$ commuting with every operator from $M$. So the
algebra generated by a Volterra algebra $A$ and its commutant has a non-zero
Volterra ideal. The interaction of these theories remained to be fruitful in
subsequent studies.

The notion of the joint spectral radius of a bounded subset $M$ in a normed
algebra $A$ was introduced by Rota and Strang \cite{RS60}. To give precise
definition, let us set $\Vert M\Vert=\sup\{\Vert a\Vert:a\in M\}$, the
\textit{norm} of $M$, and $M^{n}=\left\{  a_{1}\cdots a_{n}:a_{1},\ldots
,a_{n}\in M\right\}  $, the $n$\textit{-power} of $M$. The number
\[
\rho\left(  M\right)  :=\lim\left\Vert M^{n}\right\Vert ^{1/n}=\inf\left\Vert
M^{n}\right\Vert ^{1/n}%
\]
is called a (\textit{joint}) \textit{spectral radius} of $M$. If $\rho(M)=0$
then we say that $M$ is \textit{quasinilpotent}.


In \cite{T99} the third author, using the joint spectral radius approach,
obtained the solution of Volterra semigroup problem posed by Heydar Radjavi:
it was proved in \cite{T99} that \textit{any Volterra semigroup generates a
Vollterra algebra} and, therefore, has an IS by Lomonosov's theorem.
Further,\textbf{ } in \cite{ShT00} it was proved that\textit{ any Volterra Lie
algebra has an IS}; this result can be regarded as an infinite-dimensional
extension of Engel Theorem, playing the fundamental role in the theory of
finite-dimensional Lie algebras.

One of the main technical tools obtained and applied in \cite{ShT00} was an
infinite-dimensional extension of the Berger-Wang Theorem \cite{BW92}, a
fundamental result of the\textbf{ } finite dimensional linear dynamics
\cite{Jun09}. This theorem establishes the equality%
\begin{equation}
\rho(M)=r(M), \label{BW}%
\end{equation}
for any bounded set $M$ of matrices, where{\large
\[
r\left(  M\right)  :=\lim\sup\left\{  \rho\left(  a\right)  :a\in
M^{n}\right\}  ^{1/n};
\]
}the number $r\left(  M\right)  $ called a $BW$\textit{-radius} of $M$. In
\cite{ShT00} the equality (\ref{BW}) was proved for any precompact set $M$
\textbf{ }of compact operators on an \textit{infinite-dimensional }\textbf{
}Banach space.


To see the importance of validity of (\ref{BW}) for precompact sets of compact
operators, note that it easily implies that if $G$ is a Volterra semigroup
then $\rho(M)=0$, for each precompact $M\subset G$ (because clearly $r(M)=0$).
This result proved in \cite{T99} played a crucial role in the solution of
Volterra semigroup problem. But it should be said that the proof of (\ref{BW})
in \cite{ShT00} used the results of \cite{T99}.

Other results on invariant subspaces of operator semigroups, Lie algebras and
Jordan algebras were obtained on \textbf{ }this way in \cite{ShT00, ShT05,
KeST}{\large . }

\subsection{The generalized $BW$-formula}

To move further we have to introduce some "essential radii" $\rho_{e}(M)$,
$\rho_{f}(N)$ and $\rho_{\chi}(M)$ of a set $M$ of operators on a Banach space
$X$. They are defined in the same way as $\rho(M)$ but by using seminorms
$\Vert\cdot\Vert_{e}$, $\Vert\cdot\Vert_{f}$ and $\Vert\cdot\Vert_{\chi},$
instead of the operator norm $\Vert\cdot\Vert$.

Let $B(X)$ be the algebra of all bounded linear operators on $X$, and $K(X)$
the ideal of all compact operators. The essential norm $\Vert T\Vert_{e}$ of
an operator $T\in B(X)$ is just the norm of the image $T+K(X)$ of $T$ in the
quotient $B(X)/K(X)$; in other words%
\[
\Vert T\Vert_{e}=\inf\{\Vert T-S\Vert:S\text{ is a compact operator}\}.
\]
Similarly
\[
\Vert T\Vert_{f}=\inf\{\Vert T-S\Vert:S\text{ is a finite rank operator}\}.
\]
The Hausdorff norm $\Vert T\Vert_{\chi}$ is defined as $\chi(TX_{\odot})$, the
Hausdorff measure of noncompactness of the image of the unit ball $X_{\odot}$
of $X$ under $T$. Recall that, for any bounded subset $E$ of $X$, the value
$\chi(E)$ is equal to the infimum of such $t>0$ that $E$ has a finite $t$-net.

It is easy to check that $\Vert T\Vert_{\chi}\leq\Vert T\Vert_{e}\leq\Vert
T\Vert_{f}$ and therefore
\begin{equation}
\rho_{\chi}(M)\leq\rho_{e}(M)\leq\rho_{f}(M), \label{in}%
\end{equation}
for each bounded set $M\subset B(X)$. The number $\rho_{\chi}(M)$ is called
the \textit{Hausdorff radius}, $\rho_{e}(M)$ the \textit{essential radius},
and $\rho_{f}(M)$ the $f$-\textit{essential radius} of $M$.

In what follows, for a set $M$ in a normed algebra $A$ and a closed ideal $J$
of $A$, we write $M/J$ for the image of $M$ in $A/J$ under the canonic
quotient map $q_{J}^{{}}:A\longrightarrow A/J$:%

\[
M/J:=q_{J}^{{}}\left(  M\right)  .
\]
So we write $\rho_{e}(M)=\rho(M/K\left(  X\right)  )$. This reflects the fact
that essential radius $\rho_{e}(M)$ is the usual joint spectral radius of the
image of $M$ in the Calkin algebra $B(X)/K(X)$.

{\large I}n \cite{ShT02} the following extension of (\ref{BW}) to precompact
sets of general (not necessarily compact) operators was obtained:
\begin{equation}
\rho(M)=\max\{\rho_{\chi}(M),r(M)\}. \label{GBWF}%
\end{equation}
It was proved under assumption that $X$ is reflexive (or, more generally, that
$M$ consists of weakly compact operators). We call this equality the
\textit{generalized }$BW$\textit{-formula}{\large . }

Furthermore, in the short communication \cite{ShT01} a Banach algebraic
version of the generalized $BW$-formula was announced (see (\ref{aGBWF})
below) which, being applied to the algebra $B(X)$, shows that
\begin{equation}
\rho(M)=\max\{\rho_{e}(M),r(M)\}. \label{GBWFe}%
\end{equation}
for all Banach spaces. The proof of (\ref{GBWF}) in full generality was
firstly presented in the arXive publication \cite{ShT08}; the journal version
appeared in \cite{ShT12}.

Several months after presentation of \cite{ShT08}, I. Morris in arXive
publication \cite{Mor09} gave another proof of (\ref{GBWF}) based on the
multiplicative ergodic theorem of Tieullen \cite{Th87} and deep technique of
the theory of cohomology of dynamical systems.\textbf{ }The main result of
\cite{Mor09} establishes an equality similar to (\ref{GBWF}) for operator
valued cocycles of dynamical systems. It was also proved in \cite{Mor09}
that{\large
\begin{equation}
\rho_{\chi}(M)=\rho_{e}(M)=\rho_{f}(M) \label{main}%
\end{equation}
}for any precompact set $M\subset B(X)$. The journal publication of these
results appeared in \cite{Mor12}.

Here we give another, operator-theoretic proof of (\ref{main}) and then
discuss related Banach-algebraic results and constructions connected with the
different joint spectral radius formulas. It will be shown that topological
radicals present a convenient tool in the search of an optimal joint spectral
radius formula.

\section{Coincidence of Hausdorff and essential radii}

In this section we are going to prove the equality $\rho_{\chi}(M)=\rho
_{e}(M)$, for any precompact set in $B(X)$; the proof of the equality
$\rho_{e}(M)=\rho_{f}(M)$ will be presented in the next section.

\subsection{An estimation of the Hausdorff radius for multiplication
operators}

At the beginning we transfer some results of \textbf{ }\cite{ShT00} from
operators to elements of the Calkin algebra $B(X)/K(X)$. We use the following
link of Hausdorff norm with the Hausdorff measure of non-compactness:

\begin{lemma}
\label{chi}Let $M$ be a precompact subset of $B(X)$. Then $\chi(MW)\leq\Vert
M\Vert_{\chi}\Vert W\Vert$ for any bounded subset $W$ of $X,$ and $\Vert
M\Vert_{\chi}=\chi(MX_{\odot})$.
\end{lemma}

The inequality in Lemma \ref{chi} was obtained\textbf{ } in \cite[Lemma
5.2]{ShT00}. The equality $\left\Vert M\right\Vert ${$_{\chi}=\chi(MX_{\odot
})$ }is obvious for a finite $M\subseteq B(X)$ by definition, due to
$\chi\left(  G\cup K\right)  =\max\left\{  \chi\left(  G\right)  ,\chi\left(
K\right)  \right\}  $ for bounded subsets of $X$ If $M$ is precompact then
$\left\Vert M\right\Vert ${$_{\chi}=\sup$}$\left\{  \left\Vert N\right\Vert
{{\chi}}\text{: }N\subseteq M\text{ is finite}\right\}  $, and the result follows.

For $T\in B(X)$, let L$_{T}$ and R$_{T}$ denote the left and right
multiplication operators on $B(X)$: \textrm{L}$_{T}P=TP$ and R$_{T}P=PT$ for
each $P\in B(X)$.

For $M\subseteq B(X)$, put L$_{M}:=\{$L$_{T}$: $T\in M\}$ and R$_{M}%
:=\{$R$_{T}$: $T\in M\}$. If $M$ is a set in a Banach algebra $A$ we define
L$_{M}$ and R$_{M}$ similarly. By \cite[Lemma 2.1]{ShT12},%
\begin{equation}
r\left(  \text{L}_{M}\text{R}_{M}\right)  =r\left(  M\right)  ^{2}\text{ and
}\rho\left(  \text{L}_{M}\text{R}_{M}\right)  =\rho\left(  M\right)  ^{2}
\label{rlm}%
\end{equation}
for every bounded set $M$ in $A$.

\begin{lemma}
{\large \label{lrmk} }Let $M$ be a bounded subset of $B\left(  X\right)  $.
Then
\[
\left\Vert \mathrm{L}_{M/K\left(  X\right)  }\mathrm{R}_{M/K\left(  X\right)
}\right\Vert _{\chi}\leq16\left\Vert M\right\Vert _{\chi}\left\Vert M/K\left(
X\right)  \right\Vert .
\]

\end{lemma}

\begin{proof}
Let $T,S\in B\left(  X\right)  $. It is clear that
\begin{align*}
\left\Vert \mathrm{L}_{T/K\left(  X\right)  }\mathrm{R}_{S/K\left(  X\right)
}\right\Vert _{\chi}  &  =\chi\left(  \left(  T/K\left(  X\right)  \right)
\left(  B\left(  X\right)  /K\left(  X\right)  \right)  _{\odot}\left(
S/K\left(  X\right)  \right)  \right) \\
&  \leq\chi\left(  T\left(  B\left(  X\right)  \right)  _{\odot}S\right)
=\left\Vert \mathrm{L}_{T}\mathrm{R}_{S}\right\Vert _{\chi}.
\end{align*}
By \cite[Lemma 6.4]{ShT00}, {\large $\left\Vert \mathrm{L}_{T}\mathrm{R}%
_{S}\right\Vert _{\chi}\leq4\left(  \left\Vert T^{\ast}\right\Vert _{\chi
}\left\Vert S\right\Vert +\left\Vert S\right\Vert _{\chi}\left\Vert
T\right\Vert \right)  $ }for any $T,S\in B\left(  X\right)  ${\large .} As
$\left\Vert T^{\ast}\right\Vert _{\chi}\leq2\left\Vert T\right\Vert _{\chi}$
by \cite{GM65}, and $\left\Vert T-P\right\Vert _{\chi}=\left\Vert T\right\Vert
_{\chi}$, $\left\Vert S-F\right\Vert _{\chi}=\left\Vert S\right\Vert _{\chi}$,
for any $P,F\in K\left(  X\right)  $, we obtain that%
\begin{align*}
\left\Vert \mathrm{L}_{T/K\left(  X\right)  }\mathrm{R}_{S/K\left(  X\right)
}\right\Vert _{\chi}  &  \leq\inf_{P,F\in K\left(  X\right)  }\left\Vert
\mathrm{L}_{T-P}\mathrm{R}_{S-F}\right\Vert _{\chi}\\
&  \leq8\inf_{P,F\in K\left(  X\right)  }\left(  \left\Vert T\right\Vert
_{\chi}\left\Vert S-F\right\Vert +\left\Vert S\right\Vert _{\chi}\left\Vert
T-P\right\Vert \right) \\
&  =8\left(  \left\Vert T\right\Vert _{\chi}\left\Vert S/K\left(  X\right)
\right\Vert +\left\Vert S\right\Vert _{\chi}\left\Vert T/K\left(  X\right)
\right\Vert \right)  .
\end{align*}
Therefore{\large
\begin{align*}
\left\Vert \mathrm{L}_{M/K\left(  X\right)  }\mathrm{R}_{M/K\left(  X\right)
}\right\Vert _{\chi}  &  \leq8\sup_{T,S\in M}\left(  \left\Vert T\right\Vert
_{\chi}\left\Vert S/K\left(  X\right)  \right\Vert +\left\Vert S\right\Vert
_{\chi}\left\Vert T/K\left(  X\right)  \right\Vert \right) \\
&  \leq16\sup_{T\in M}\left\Vert T\right\Vert _{\chi}\sup_{S\in M}\left\Vert
S/K\left(  X\right)  \right\Vert =16\left\Vert M\right\Vert _{\chi}\left\Vert
M/K\left(  X\right)  \right\Vert .
\end{align*}
}
\end{proof}

\subsection{Semigroups in the Calkin algebra}

Let $M\subseteq B\left(  X\right)  $, and let SG$\left(  M\right)  $ be the
semigroup generated by $M$. The same notation is used if $M$ is a subset of an
arbitrary Banach algebra.

\begin{proposition}
\label{mpre} Let $M$ be a precompact subset of $B\left(  X\right)  $. If
$\mathrm{SG}\left(  M/K\left(  X\right)  \right)  $ is bounded and
$\rho_{\mathsf{\chi}}\left(  M\right)  <1$ then $\mathrm{SG}\left(  M/K\left(
X\right)  \right)  $ is precompact.
\end{proposition}

\begin{proof}
Let $G_{n}=\cup\left\{  M^{k}/K\left(  X\right)  :k>n\right\}  $ for each
$n\geq0$. As $L_{M^{k}/K\left(  X\right)  }R_{M^{k}/K\left(  X\right)  }$ is a
precompact set in $B\left(  B\left(  X\right)  /K\left(  X\right)  \right)  $,
then, by Lemmas \ref{chi} and \ref{lrmk},
\begin{align*}
\chi\left(  G_{2k}\right)   &  =\chi\left(  \left(  M^{k}/K\left(  X\right)
\right)  G_{0}\left(  M^{k}/K\left(  X\right)  \right)  \right)  =\chi\left(
\mathrm{L}_{M^{k}/K\left(  X\right)  }\mathrm{R}_{M^{k}/K\left(  X\right)
}G_{0}\right) \\
&  \leq\left\Vert \mathrm{L}_{M^{k}/K\left(  X\right)  }\mathrm{R}%
_{M^{k}/K\left(  X\right)  }\right\Vert _{\chi}\left\Vert G_{0}\right\Vert
\leq16\left\Vert M^{k}\right\Vert _{\chi}\left\Vert M^{k}/K\left(  X\right)
\right\Vert \left\Vert G_{0}\right\Vert \\
&  \leq\left(  16\left\Vert G_{0}\right\Vert ^{2}\right)  \left\Vert
M^{k}\right\Vert _{\chi}.
\end{align*}
As $\rho_{\mathsf{\chi}}\left(  M\right)  <1$, there is $m>0$ such that
$\left\Vert M^{m}\right\Vert _{\chi}<1/2$. Then for $n>2km$, we have that
\[
\chi\left(  G_{n}\right)  \leq\chi\left(  G_{2km}\right)  \leq\left(
16\left\Vert G_{0}\right\Vert ^{2}\right)  \left(  1/2\right)  ^{k}%
\rightarrow0\text{ under }k\rightarrow\infty.
\]
This shows that $\chi\left(  G_{n}\right)  \rightarrow0$ under $n\rightarrow
0$. As SG$\left(  M/K\left(  X\right)  \right)  \backslash G_{n}$ is
precompact,
\[
\chi\left(  \text{SG}\left(  M/K\left(  X\right)  \right)  \right)
=\chi\left(  G_{n}\right)
\]
for every $n$. Therefore $\chi\left(  \text{SG}\left(  M/K\left(  X\right)
\right)  \right)  =0$, i.e. SG$\left(  M/K\left(  X\right)  \right)  $ is precompact.
\end{proof}

Let $A$ be a Banach algebra and $M\subseteq$ $A$. Let LIM$\left(  M\right)  $
be the set of limit points of all sequences $\left(  a_{k}\right)  $ with
$a_{k}\in M^{n_{k}}$, $n_{k}\rightarrow\infty$ when $k\rightarrow\infty$. It
follows from \cite[Corollary 6.12]{ShT00} that\textit{ if }$\rho\left(
M\right)  =1$\textit{ and }SG$\left(  M\right)  $\textit{ is precompact then
}LIM$\left(  M\right)  =$ LIM$\left(  M\right)  ^{2}$\textit{ and it has a
non-zero idempotent.} We will use this fact in the proof of Theorem \ref{morr}
(the part Case 1).

An element $a\in A$ is called $n$\textit{-leading} for $M$ if $a\in M^{n}$ and
$\left\Vert a\right\Vert \geq\left\Vert \cup_{k<n}M^{k}\right\Vert $; a
sequence $\left(  a_{k}\right)  \subseteq A$ is called \textit{leading} for
$M$, if $a_{k}$ is $n_{k}$-leading for $M$, where $n_{k}\rightarrow\infty$,
and $\left\Vert a_{k}\right\Vert \rightarrow\infty$ under $k\rightarrow\infty$.

Let ld$^{n}\left(  M\right)  $ be the set of all $n$-leading elements for $M$,
ld$\left(  M\right)  =\cup_{n\geq2}$ld$^{n}\left(  M\right)  $ and
ld$_{\left[  1\right]  }\left(  M\right)  =\left\{  a/\left\Vert a\right\Vert
\text{: }a\in\text{ld}\left(  M\right)  \right\}  $.

\begin{lemma}
\label{ldpre} Let $M$ be a precompact set of $B\left(  X\right)  $. If
$\left\Vert M^{m}\right\Vert _{\chi}\left\Vert M^{m}/K\left(  X\right)
\right\Vert <1$, for some $m>0$ then $\mathrm{ld}_{\left[  1\right]  }\left(
M/K\left(  X\right)  \right)  $ is precompact.
\end{lemma}

\begin{proof}
Let $G_{n}=\left\{  a/\left\Vert a\right\Vert \text{: }a\in\text{ld}%
^{i}\left(  M/K\left(  X\right)  \right)  ,i\geq n\right\}  $ for any $n>0$.
Let $n=2km+j$, where $0\leq j<2m$. Then, for $N=M^{m}$ and $B_{\left(
1\right)  }=\left(  B\left(  X\right)  /K\left(  X\right)  \right)  _{\odot}$,
we obtain that%
\begin{equation}
G_{n}\subseteq\left(  N^{k}/K\left(  X\right)  \right)  B_{\left(  1\right)
}\left(  N^{k}/K\left(  X\right)  \right)  . \label{gn}%
\end{equation}
Indeed, if $T/K\left(  X\right)  \in$ ld$^{i}\left(  M/K\left(  X\right)
\right)  $ where $i\geq n$, then
\[
T=T_{1}T_{2}T_{3}/K\left(  X\right)
\]
for some $T_{1}/K\left(  X\right)  ,T_{3}/K\left(  X\right)  \in
N^{k}/K\left(  X\right)  $ and $T_{2}/K\left(  X\right)  \in M^{i-2km}%
/K\left(  X\right)  $. As $T/K\left(  X\right)  $ is an $i$-leading element
for $M/K\left(  X\right)  ,$ then
\[
\left\Vert T_{2}/K\left(  X\right)  \right\Vert \leq\left\Vert T/K\left(
X\right)  \right\Vert .
\]
This proves $\left(  \ref{gn}\right)  $.

Let $t=\left\Vert N\right\Vert _{\chi}\left\Vert N/K\left(  X\right)
\right\Vert $. As $N^{k}/K\left(  X\right)  $ is a precompact set, we get from
Lemmas \ref{chi} and \ref{lrmk} that
\begin{align*}
\chi\left(  G_{n}\right)   &  \leq\chi\left(  \text{L}_{N^{k}/K\left(
X\right)  }\text{R}_{N^{k}/K\left(  X\right)  }B_{\left(  1\right)  }\right)
\leq\left\Vert \text{L}_{N^{k}/K\left(  X\right)  }\text{R}_{N^{k}/K\left(
X\right)  }\right\Vert _{\chi}\\
&  \leq16\left\Vert N^{k}\right\Vert _{\chi}\left\Vert N^{k}/K\left(
X\right)  \right\Vert \leq16\left(  \left\Vert N\right\Vert _{\chi}\left\Vert
N/K\left(  X\right)  \right\Vert \right)  ^{k}=16t^{k}%
\end{align*}
whence $\chi\left(  G_{n}\right)  \rightarrow 0$ under $n \rightarrow\infty$. As
ld$_{\left[  1\right]  }\left(  M/K\left(  X\right)  \right)  \backslash
G_{n}$ is precompact,
\[
\chi\left(  \text{ld}_{\left[  1\right]  }\left(  M/K\left(  X\right)
\right)  \right)  =\chi\left(  G_{n}\right)
\]
for every $n$. Therefore $\chi\left(  \text{ld}_{\left[  1\right]  }\left(
M/K\left(  X\right)  \right)  \right)  =0$, i.e., ld$_{\left[  1\right]
}\left(  M/K\left(  X\right)  \right)  $ is a precompact set.
\end{proof}

\subsection{Hausdorff radius equals essential radius}

\begin{theorem}
\label{morr} Let $M$ be a precompact subset of $B\left(  X\right)  $. Then
\begin{equation}
\rho_{e}\left(  M\right)  =\rho_{\mathsf{\chi}}\left(  M\right)  .\label{mor}%
\end{equation}

\end{theorem}

\begin{proof}
Let $\rho_{e}\left(  M\right)  =1$. Assume, aiming at the contrary, that
$\rho_{\mathsf{\chi}}\left(  M\right)  <1$. We consider two cases.

\textbf{Case 1.} {\large $\mathrm{SG}\left(  M/K\left(  X\right)  \right)  $
}\textit{is bounded. }

By Proposition \ref{mpre}, SG$\left(  M/K\left(  X\right)  \right)  $ is
precompact. As $\rho\left(  M/K\left(  X\right)  \right)  =1$, then
$\mathrm{LIM}\left(  M/K\left(  X\right)  \right)  $ has a non-zero idempotent
by \cite[Corollary 6.12]{ShT00}. On the other hand, let $T\in$ SG$\left(
M\right)  $ be an arbitrary operator such that $T/K\left(  X\right)
\in\mathrm{LIM}\left(  M/K\left(  X\right)  \right)  $. Then there is a
sequence $\left(  T_{k}\right)  $ with $T_{k}/K\left(  X\right)  \in\left(
M/K\left(  X\right)  \right)  ^{n_{k}}$ for $n_{k}\rightarrow\infty$ and
$T_{k}/K\left(  X\right)  \rightarrow T/K\left(  X\right)  $ under
$k\rightarrow\infty$. Hence $\left\Vert T_{k}-T\right\Vert _{\chi}%
\rightarrow0$ under $k\rightarrow\infty$. As $q:=\rho_{\mathsf{\chi}}\left(
M\right)  <1$ then $\left\Vert T_{k}\right\Vert _{\chi}\leq q^{n_{k}%
}\rightarrow0$ under $n_{k}\rightarrow\infty$. So $T$ is a compact operator.
Hence $\mathrm{LIM}\left(  M/K\left(  X\right)  \right)  =\left(  0\right)  $,
a contradiction. This shows that $\rho_{\mathsf{\chi}}\left(  M\right)
=\rho_{e}\left(  M\right)  $ holds in Case 1.

\textbf{Case 2.} {\large $\mathrm{SG}\left(  M/K\left(  X\right)  \right)  $
}\textit{is not bounded. }

It follows easily from definition that in this case there exists a leading
sequence for $M/K\left(  X\right)  $. Let $\left(  T_{k}/K\left(  X\right)
\right)  _{k=1}^{\infty}$ be such a sequence. For brevity, set $a_{k}%
=T_{k}/K\left(  X\right)  $ for each $k$. Then
\[
G:=\left\{  a_{k}/\left\Vert a_{k}\right\Vert \text{: }k\in\mathbb{N}\right\}
\subseteq\text{ ld}_{\left[  1\right]  }\left(  M/K\left(  X\right)  \right)
.
\]
Let $\rho_{\mathsf{\chi}}\left(  M\right)  =t_{1}<1$ and $t_{1}<t_{2}<1$. It
follows from the condition $\rho_{e}\left(  M\right)  =1$ that, for any
$\varepsilon>0$ with $t_{2}\left(  1+\varepsilon\right)  <1$, there is
$n_{1}>0$ such that $\left\Vert M^{n}/K\left(  X\right)  \right\Vert
^{1/n}<1+\varepsilon$ for all $n>n_{1}$, and also there is $n_{2}>0$ such that
$\left\Vert M^{n}\right\Vert _{\chi}^{1/n}<t_{2}$. Then
\[
\left\Vert M^{n}\right\Vert _{\chi}\left\Vert M^{n}/K\left(  X\right)
\right\Vert <\left(  t_{2}\left(  1+\varepsilon\right)  \right)  ^{n}<1
\]
for any $n>\max\left\{  n_{1},n_{2}\right\}  $. By Lemma \ref{ldpre}, $G$ is
precompact. Let $b:=S/K\left(  X\right)  $ be a limit point of $G$. One may
assume that
\[
\left\Vert b-a_{k}/\left\Vert a_{k}\right\Vert \right\Vert \rightarrow0\text{
under }k\rightarrow\infty.
\]
It is clear that $\left\Vert b\right\Vert =1$. We have%
\begin{align}
\left\Vert S\right\Vert _{\chi}  &  \leq\left\Vert S-T_{k}/\left\Vert
a_{k}\right\Vert \right\Vert _{\chi}+\left\Vert T_{k}/\left\Vert
a_{k}\right\Vert \right\Vert _{\chi}\label{sx}\\
&  \leq\left\Vert b-a_{k}/\left\Vert a_{k}\right\Vert \right\Vert +\left\Vert
T_{k}\right\Vert _{\chi}/\left\Vert a_{k}\right\Vert .\nonumber
\end{align}

As $\rho_{\mathsf{\chi}}\left(  M\right)  <1$, $\left\{  \left\Vert
T_{k}\right\Vert _{\chi}\text{: }k\in\mathbb{N}\right\}  $ is a bounded set.
As $\left\Vert a_{k}\right\Vert \longrightarrow_{k}\infty,$ we get $\left\Vert
T_{k}\right\Vert _{\chi}/\left\Vert a_{k}\right\Vert \rightarrow_{k}0.$ We
obtain from $\left(  \ref{sx}\right)  $ that $\left\Vert S\right\Vert _{\chi
}=0$, i.e., $S$ is a compact operator. Hence immediately $b=0$, a
contradiction. Thus $\rho_{e}\left(  M\right)  =\rho_{\mathsf{\chi}}\left(
M\right)  $ in any case.
\end{proof}



\section{Banach-algebraic approach to the joint spectral radius formulas}

\subsection{$BW$-ideals}

Now we present the Banach-algebraic approach to the formulas for the joint
spectral radius. Let us consider a Banach algebra $A$ instead of $B\left(
X\right)  $. Let $BW\left(  A\right)  $ denote the set of all closed ideals
$J$ of $A$ such that
\begin{equation}
\rho(M)=\max\{\rho(M/J),r(M)\}\text{ for all precompact }M\subseteq A.
\label{GBWJ}%
\end{equation}

The ideals $J$ for which (\ref{GBWJ}) holds are called $BW$\textit{-ideals}.
Clearly, if $I\subset J$, $J\in BW(A)$ then $I\in BW(A)$. It is known that $BW(A)$
has maximal elements; moreover it was proved in \cite[Lemma 5.2]{ShT12} that
if $J=\overline{\cup{J_{\lambda}}}$ where $(J_{\lambda})$ is a linearly
ordered set of $BW$-ideals of $A$ then $J\in BW(A)$.

Let us call an increasing transfinite sequence $(J_{\alpha})_{\alpha\leq
\gamma}$ of closed ideals in a Banach algebra $A$ an\textit{ increasing
transfinite\textbf{ } chain of closed ideals} if $J_{\beta}=\overline
{\cup_{\alpha<\beta}J_{\alpha}}$ for any limit ordinal $\beta\leq\gamma$, and
a decreasing transfinite sequence $(I_{\alpha})_{\alpha\leq\gamma}$ -- a
\textit{decreasing transfinite\textbf{ }chain of closed ideals} if $I_{\beta
}=\cap_{\alpha<\beta}I_{\alpha}$ for any limit ordinal $\beta\leq\gamma$.

By \cite{ShT12}, if $I\subset J$ are closed ideals of $A$, $I\in BW(A)$ and
$J/I\in BW(A/I)$ then $J\in BW(A)$. This implies the transfinite stability for
$BW$-ideals.

\begin{proposition}
\label{ESBW} If in increasing transfinite\textbf{ } chain $(J_{\alpha
})_{\alpha\leq\gamma}$ of closed ideals in a Banach algebra $A$ the ideal
$J_{0}$ belongs to $BW(A)$ and $J_{\alpha+1}/J_{\alpha}\in BW(A/J_{\alpha})$,
for all $\alpha$, then $J_{\gamma}\in BW(A)$.
\end{proposition}

Every $BW$-ideal $J$ of a Banach algebra turns out to be a
\textit{Berger-Wang algebra}in the sense that the equality%
\begin{equation}
\rho\left(  M\right)  =r\left(  M\right)  \label{bwa}%
\end{equation}
holds for any precompact set $M$ of $J$. It follows from (\ref{bwa}) and
\cite[Proposition 3.5]{T85} that \textit{ every semigroup consisting of
quasinilpotent elements of a Berger-Wang algebra generates a finitely
quasinilpotent subalgebra.}

Since the Jacobson radical $\mathrm{Rad}\left(  A\right)  $ of every Banach
algebra $A$ consists of quasinilpotents, then for a Berger-Wang Banach algebra
$A$, $\mathrm{Rad}\left(  A\right)  $ \textit{is compactly quasinilpotent,
i.e., }$\rho\left(  M\right)  =0$ for any precompact set $M$ of
$\mathit{\mathrm{Rad}}\left(  A\right)  $.

\subsection{First Banach-algebraic formulas for the joint spectral radius}

A natural analogue of compact operators in the Banach algebra context was
proposed by K. Vala \cite{Vala64} who proved that the map $T\mapsto
S_{1}TS_{2}$ on the algebra $B(X)$ is compact if and only the operators
$S_{1}$ and $S_{2}$ are compact. So an element $a$ of a normed algebra $A$ is
called\textit{ compact} (\textit{finite rank}) if the operator L$_{a}$R$_{a}$:
$x\mapsto axa$ on $A$ is compact (finite rank). A normed algebra $A$ is called
\textit{bicompact} if all operators L$_{a}$R$_{b}$: $x\mapsto axb$ ($a,b\in
A$) are compact. An ideal of $A$ is called \textit{bicompact} if it is
bicompact as a normed algebra.

It follows from \cite[Corollary 4.8]{ShT12} that for every bicompact ideal $J$
of $A$ the equality (\ref{GBWJ}) holds. Since, by \cite{Vala64}, $K\left(
X\right)  $ is a bicompact ideal of $B\left(  X\right)  $, this result widely
extends the generalized $BW$-formula (\ref{GBWFe}) (which is the same as
(\ref{GBWF}) \ by\textbf{ }virtue of Theorem \ref{morr}\textbf{)}.

A normed algebra $A$ is called \textit{hypocompact} (\textit{hypofinite}) if
every non-zero quotient $A/J$ has a non-zero compact (finite rank) element. An
ideal is\textit{ hypocompact} (\textit{hypofinite}) if it is hypocompact
(hypofinite) as a normed algebra.

Each bicompact algebra is hypocompact, and any hypocompact ideal can be
composed from bicompact blocks:

\begin{proposition}
\cite[Proposition 3.8]{ShT12}\label{phc1} For any hypocompact closed ideal $I$
of a Banach algebra $A,$ there is a transfinite\textbf{ }increasing chain
$(J_{\alpha})_{\alpha\leq\gamma}$ of closed ideals of $A$ such that
$J_{1}=\left(  0\right)  $ and $J_{\gamma}=I,$ and every quotient space
$J_{\alpha+1}/J_{\alpha}$ is a bicompact ideal of $A/J_{\alpha}$.
\end{proposition}


\begin{theorem}
\cite[Theorem 4.11]{ShT12}\label{hcf} The formula (\ref{GBWJ}) holds for every
hypocompact closed ideal $J$ of $A$.
\end{theorem}

Indeed, as every closed bicompact ideal of a Banach algebra $A$ is a
$BW$-ideal, the result follows from Propositions \ref{ESBW} and \ref{phc1}.

A Banach algebra is called \textit{scattered} if its elements have countable
spectra. It follows from \cite[Theorem 8.15]{ShT14} that every hypocompact
algebra is scattered.

\subsection{Compact quasinilpotence, and coincidence of essential and
$f$-essential joint spectral radii}

Recall that a Banach algebra $A$ is \textit{compactly quasinilpotent} if
$\rho(M)=0$ for any precompact subset $M$ of $A$.

The following result shows that any compactly quasinilpotent ideal can be
considered as inessential when one calculates the joint spectral radius.

\begin{theorem}
\cite[Theorem 4.18]{ShT05r}{\large \label{iness} }$\rho\left(  M\right)
=\rho\left(  M/J\right)  $ for each compactly quasinilpotent ideal and
precompact set $M\subset A$.
\end{theorem}

In particular all compactly quasinilpotent ideals are BW-ideals.

\begin{theorem}
{\large \cite[Theorem 3.14]{ShT12}\label{inf1} }If a Banach algebra $A$ is
hypocompact and consists of quasinilpotents then it is compactly quasinilpotent.
\end{theorem}

The following result shows that the reverse inclusion fails.

\begin{proposition}
\label{radNOTcomp} There are compactly quasinilpotent Banach algebras without
non-zero hypocompact ideals.
\end{proposition}

\begin{proof}
Let $V$ be the algebra $\ell^{1}(w)$, where the weight $w=(w_{k}%
)_{k=1}^{\infty}$ satisfies the condition
\begin{equation}
\lim\left(  w_{k+1}/w_{k}\right)  =0 \label{w}%
\end{equation}
(for instance, one can take $w_{k}=1/k^{k}$). It follows easily from (\ref{w})
that such a weight is radical, that is $\lim_{k\to\infty}w_{k}^{1/k} = 0$.
Therefore all elements of $V$ are quasinilpotent.

Let $A$ be the projective tensor product $V\widehat{\otimes}B$ of $V$ and any
commutative Banach algebra $B$ without non-zero compact elements (for
instance, one may take for $B$ the algebra $C\left[  0.1\right]  $ of
continuous functions on $\left[  0.1\right]  \subseteq\mathbb{R}$).

Let us write elements $v\in V$ as
\begin{equation}
v=\sum_{k=1}^{\infty}\lambda_{k}e_{k}, \label{v}%
\end{equation}
where $e_{k}$ is the sequence $(\alpha_{1},\alpha_{2},...)$ with $\alpha
_{i}=1$ if $i=k$, and $0$ otherwise. It follows that $e_{k}\neq0$ for all $k$,
so that $V$ has no non-zero nilpotents. Indeed, if $v^{m}=0$ and $\lambda_{n}$
is the first non-zero coefficient in the expansion (\ref{v}) then clearly
$e_{mn}=0$, a contradiction.

To see that the algebra $V$ is compact, note that the set of all compact
elements in any Banach algebra is closed and with each element contains the
algebra generated by it. So it suffices to show that the element $e_{1}$ is
compact, because $V$ is topologically generated by $e_{1}$.

Let $V_{\odot}$ be the unit ball of $V$:
\[
V_{\odot}=\left\{  \sum_{k=1}^{\infty}\lambda_{k}e_{k}:\sum_{k=1}^{\infty
}|\lambda_{k}|w_{k}\leq1\right\}  .
\]
We are going to show that L$_{e_{1}}$ is a compact operator. For this, it
suffices to show that, for each $\varepsilon>0,$ the set $e_{1}V_{\odot}$
contains a finite $\varepsilon$-net.

For each $n$, let $P_{n}$ be the natural projection on the linear span of
$\{e_{1},...,e_{n}\}$, and let $K_{n}=P_{n}V_{\odot}$ and $K_{n}^{\bot}%
=(1-P_{n})V_{\odot}$. Then
\[
e_{1}V_{\odot}\subset e_{1}K_{n}+e_{1}K_{n}^{\bot}.
\]
The set $e_{1}K_{n}$ is compact for each $n$, so in any case it contains a
finite $(\varepsilon/2)$-net. Now it suffices to show that $\Vert e_{1}%
a\Vert\leq\varepsilon/2$ for each $a\in K_{n}^{\bot}$ if $n$ is sufficiently large.

By (\ref{w}), there is $n$ such that $w_{k+1}<\varepsilon w_{k}/2$ for all
$k\geq n$. Then for each $a=\sum_{k=n+1}^{\infty}\lambda_{k}e_{k}\in K_{n}^{\bot}$,
we have
\begin{align*}
\Vert e_{1}a\Vert &  =\left\Vert \sum_{k=n+1}^{\infty}\lambda_{k}%
e_{k+1}\right\Vert =\sum_{k=n+1}^{\infty}|\lambda_{k}|w_{k+1}<\varepsilon
\left(  \sum_{k=n+1}^{\infty}|\lambda_{k}|w_{k}\right)  /2=\varepsilon\Vert
a\Vert/2\\
&  \leq\varepsilon/2.
\end{align*}
Thus $e_{1}V_{\odot}$ contains a finite $\varepsilon$-net for every
$\varepsilon>0$, i.e., $e_{1}V_{\odot}$ is a compact set, whence $V$ is a
compact algebra consisting of quasinilpotents.

Applying Theorem \ref{inf1} we see that the algebra $V$ is compactly
quasinilpotent. By \cite[Theorem 4.29]{ShT05r}, the same is true for the
tensor product of $V$ and any Banach algebra. Thus the algebra $A$ is
compactly quasinilpotent. It remains to show that $A$ is not hypocompact. Each
element of $A$ has the form
\[
a=\sum_{k=1}^{\infty}e_{k}\otimes b_{k},\text{ where }\sum_{k=1}^{\infty}\Vert
b_{k}\Vert w_{k}<\infty.
\]

Suppose that an element $c\in A$ is compact. Since $A$ is commutative, the
operator L$_{c^{2}}=$ L$_{c}$R$_{c}$ is compact. Thus setting $a=c^{2}$ we
have that the set $aA_{\odot}$ is precompact. Let $B_{\odot}$ be the unit ball
of $B$. If $aA_{\odot}$ is a precompact subset of $A$ then, in particular, the
set
\[
\left\{  a(e_{1}\otimes b)\text{: }b\in B_{\odot}\right\}  =\left\{
\sum_{k=1}^{\infty}e_{k+1}\otimes b_{k}b\text{: }b\in B_{\odot}\right\}
\]
is precompact. In particular, all sets
\[
E_{k}:=\{e_{k+1}\otimes b_{k}b\text{: }b\in B_{\odot}\}
\]
are precompact because the natural projection of $V\widehat{\otimes}B$ onto the subspace $e_j{\otimes}B$ is bounded. Each set $E_{k}$ is homeomorphic to $b_{k}B_{\odot}$ whence
$b_{k}$ is a compact element of $B$. Since $B$ has no non-zero compact
elements, $b_{k}=0$ for any $k>0$, whence $a=0$, i.e. $c^{2}=0$.

Let us show that $c=0$. Indeed, if $c\neq0$ let
\[
c=\sum_{k=1}^{\infty}e_{k}\otimes d_{k},
\]
and let $d_{m}$ be the first non-zero element among all $d_{k}$. Then
\[
0=c^{2}=\sum_{k=2m}^{\infty}e_{k}\otimes\sum_{i+j=k}d_{i}d_{j}\text{ whence
}d_{m}^{2}=0.
\]
Therefore $d_{m}$ is a compact element of $B$ (since $B$ is commutative, $d_mbd_m = d_m^2b = 0$, for all $b\in B$). Since $B$ has no non-zero
compact elements, $d_{m}=0$, a contradiction.

We proved that $A$ has no non-zero compact elements. It follows that $A$ has
no bicompact and hypocompact ideals.
\end{proof}

\begin{theorem}
\label{mainRad} Let $A$ be a Banach algebra. Then there are the largest
hypocompact ideal $\mathcal{R}_{\mathrm{hc}}\left(  A\right)  $, the largest
hypofinite ideal $\mathcal{R}_{\mathrm{hf}}\left(  A\right)  $, the largest
compactly quasinilpotent ideal $\mathcal{R}_{\mathrm{cq}}\left(  A\right)  $
and the largest scattered ideal $\mathcal{R}_{\mathrm{sc}}\left(  A\right)  $.
\end{theorem}

For the proofs see \cite[Corollary 3.10]{ShT12}, \cite[Theorem 4.18]{ShT05r},
and \cite[Theorem 8.10]{ShT14}.

Now we return to our initial problem.

\begin{theorem}
\label{morr2} Let $M$ be a precompact subset of $B\left(  X\right)  $. Then
\begin{equation}
\rho_{f}\left(  M\right)  =\rho_{e}\left(  M\right)  . \label{mor2}%
\end{equation}

\end{theorem}

\begin{proof}
As $K\left(  X\right)  $ is a bicompact algebra by \cite{Vala64}, the algebra
$K(X)/\overline{F(X)}$ is also bicompact. As spectral projections of compact
operators are in $F(X)$, it is easy to see that $K(X)/\overline{F(X)}$
consists of quasinilpotents. Then it is compactly quasinilpotent by Theorem
\ref{inf1}. Therefore $K\left(  X\right)  /\overline{F\left(  X\right)  }$ is
a compactly quasinilpotent ideal of $B\left(  X\right)  /\overline{F\left(
X\right)  }$. Using Theorem \ref{iness} applied to $J=K\left(  X\right)
/\overline{F\left(  X\right)  }$, we have that
\begin{align*}
\rho_{f}\left(  M\right)   &  =\rho\left(  M/\overline{F\left(  X\right)
}\right)  =\rho\left(  \left(  M/\overline{F\left(  X\right)  }\right)
/\left(  K\left(  X\right)  /\overline{F\left(  X\right)  }\right)  \right) \\
&  =\rho\left(  M/K\left(  X\right)  \right)  =\rho_{e}\left(  M\right)
\end{align*}
for a precompact subset $M$ of $B\left(  X\right)  $.
\end{proof}

\subsection{The largest BW-ideal problem and topological radicals}

Let $A$ be a Banach algebra. As it was already noted, the set of all
$BW$-ideals has maximal elements. However, it is not known whether
$\overline{I+J}\in BW(A)$ if $I,J\in BW(A)$. So the problem of existence of
the largest $BW$-ideal is open.

On the other hand, the largest BW-ideal problem disappears if one only
consideres\textbf{ } ideals defined by some natural properties --- as, for
example, the ideals $\mathcal{R}_{\mathrm{hc}}\left(  A\right)  $,
$\mathcal{R}_{\mathrm{hf}}\left(  A\right)  $, $\mathcal{R}_{\mathrm{cq}%
}\left(  A\right)  $ and $\mathcal{R}_{\mathrm{sc}}\left(  A\right)  $ defined
in Theorem \ref{mainRad}. To formulate this precisely we turn to the theory of
topological radicals. We recall some definitions and results of this theory; a
reader can refer to the works \cite{D97, ShT05r, KShT09, ShT10, KShT12, ShT12,
ShT14, CT16} for additional information.

In what follows the term \textit{ideal} will mean \textit{a two-sided ideal}.
In general, radicals can be defined on classes of rings and algebras; the
topological radicals are defined on classes of normed algebras.
A radical is an \textit{ideal map},\textit{ }i.e., a map that assigns to each
algebra its ideal, while a topological radical is a \textit{closed ideal map},
it assigns to a normed algebra its closed ideal. In correspondence with our
subject here we restrict our attention to the class of all Banach algebras.

We begin with the most important and convenient class of topological radicals.
A\textit{ hereditary topological radical} on the class of all Banach algebras
is a closed ideal map $\mathcal{P}$ which assigns to each Banach algebra $A$ a
closed two-sided ideal $\mathcal{P}(A)$ of $A$ and satisfies the following conditions:

\begin{itemize}
\item[$\left(  \mathrm{H1}\right)  $] $f(\mathcal{P}(A))\subset\mathcal{P}(B)$
for a continuous surjective homomorphism $f:A\longrightarrow B$;

\item[$\left(  \mathrm{H2}\right)  $] $\mathcal{P}(A/\mathcal{P}(A))=\left(
0\right)  $;

\item[$\left(  \mathrm{H3}\right)  $] $\mathcal{P}(J)=J\cap\mathcal{P}(A)$ for
any ideal $J$ of $A$.
\end{itemize}

It can be seen from $\left(  \mathrm{H2}\right)  $ that every radical
$\mathcal{P}$ accumulates some special property in the ideal $\mathcal{P}%
\left(  A\right)  $ of an algebra $A$ which is called \textit{ the
$\mathcal{P}$-radical }of $A$.

For the proof of the following theorem see \cite[Theorem 4.25]{ShT05r},
{\large \cite[Theorems 3.58 and 3.59]{ShT10}, } {\large \cite[Section
8]{ShT14}}.

\begin{theorem}
The maps $\mathcal{R}_{\mathrm{cq}}$\emph{: }$A\longmapsto\mathcal{R}%
_{\mathrm{cq}}\left(  A\right)  $, $\mathcal{R}_{\mathrm{hc}}$\emph{:
}$A\longmapsto\mathcal{R}_{\mathrm{hc}}\left(  A\right)  $, $\mathcal{R}%
_{\mathrm{hf}}$\emph{: }$A\longmapsto\mathcal{R}_{\mathrm{hf}}\left(
A\right)  $ and $\mathcal{R}_{\mathrm{sc}}$: $A\longmapsto\mathcal{R}%
_{\mathrm{sc}}\left(  A\right)  $ are hereditary topological radicals.
\end{theorem}

The maps $\mathcal{R}_{\text{hc}}$, $\mathcal{R}_{\text{h}\mathrm{f}}$,
$\mathcal{R}_{\text{cq}}$ and $\mathcal{R}_{\text{sc}}$ are called the
\textit{hypocompact, hypofinite, compactly quasinilpotent} and
\textit{scattered radical} respectively.

It follows immediately from Axiom (H3) that hereditary radicals satisfy the
conditions:\smallskip

\begin{itemize}
\item[$\left(  \mathrm{I1}\right)  $] $\mathcal{P}(\mathcal{P}(A))=\mathcal{P}%
(A)$;

\item[$\left(  \mathrm{I2}\right)  $] $\mathcal{P}(J)$ of an ideal $J$ of $A$
is an ideal of $A$ which is contained in the radical $\mathcal{P}%
(A)$.\smallskip
\end{itemize}

If a closed ideal map $\mathcal{P}$ on the class of all Banach algebras
satisfies (H1), (H2) and, instead of (H3), also (I1) and (I2) then
$\mathcal{P}$ is called a \textit{topological radical} (see \cite{D97}).

If an ideal map [a closed ideal map] $\mathcal{P}$ satisfies (H1), it is
called a \textit{preradical }[a \textit{topological preradical}].

A closed ideal map $\mathcal{P}$ is called an \textit{under topological
radical} (UTR) if it satisfies all axioms of topological radicals, besides
possibly of $\left(  \mathrm{H2}\right)  $, and an \textit{over topological
radical} (OTR) if it satisfies all axioms, apart from possibly of $\left(
\mathrm{I1}\right)  $ (see \cite[Definition 6.2]{D97})).

Given a preradical $\mathcal{P},$ an algebra $A$ is called $\mathcal{P}%
$\textit{-radical} if {\large $A=\mathcal{P}(A)$,} and $\mathcal{P}%
$\textit{-semisimple} if {\large $\mathcal{P}$}$(A)=0$. It follows from
$\left(  \mathrm{H1}\right)  $ for a topological preradical $\mathcal{P}$ that
$\mathcal{P}$-radical and $\mathcal{P}$-semisimple algebras are invariant with
respect to topological isomorphisms.

Let $\mathcal{P}$ be a topological radical. It follows easily from the
definition that quotients of $\mathcal{P}$-radical algebras are $\mathcal{P}%
$-radical, and ideals of {\large $\mathcal{P}$}-semisimple algebras are
{\large $\mathcal{P}$-}semisimple. Moreover, the class of all $\mathcal{P}%
$-radical ($\mathcal{P}$-semisimple) algebras is stable with respect to
extensions: \textit{If }$J$ is a $\mathcal{P}$\textit{-radical }($\mathcal{P}%
$-semisimple)\textit{ ideal} of $A$ \textit{and the quotient} $A/J$\textit{ is
}$\mathcal{P}$\textit{-radical }($\mathcal{P}$-semisimple)\textit{ then }%
$A$\textit{ itself is also }$\mathcal{P}$\textit{-radical }($\mathcal{P}%
$-semisimple)\textit{.}

The proof of the following properties \textit{of transfinite stability} can be
found in \cite[Theorem 4.18]{ShT14}.

\begin{proposition}
Let $\mathcal{P}$ be a topological radical, $A$ a Banach algebra, and let
$(I_{\alpha})_{\alpha\leq\gamma}$ and $(J_{\alpha})_{\alpha\leq\gamma}$ be
decreasing and increasing transfinite\textbf{ }chains of closed ideals of $A$. Then

\begin{enumerate}
\item If $A/I_{1}$ and all quotients $I_{\alpha}/I_{\alpha+1}$ are
$\mathcal{P}$-semisimple then $A/I_{\gamma}$ is $\mathcal{P}$-semisimple;

\item If $J_{1}$ and all quotients $J_{\alpha+1}/J_{\alpha}$ are $\mathcal{P}%
$-radical then $J_{\gamma}$ is $\mathcal{P}$-radical.
\end{enumerate}
\end{proposition}


\section{Around joint spectral radius formulas and radicals}

\subsection{Comparison of joint spectral radius formulas}

It follows from Theorem \ref{hcf} that for any Banach algebra $A$ and
precompact set $M\subseteq A$, the equality%
\begin{equation}
\rho(M)=\max\{\rho(M/\mathcal{R}_{\mathrm{hc}}(A)),r(M)\} \label{aGBWF}%
\end{equation}
holds. Since $\mathcal{R}_{\mathrm{hf}}(A)\subseteq\mathcal{R}_{\mathrm{hc}%
}(A)$, we certainly have%
\begin{equation}
\rho(M)=\max\{\rho(M/\mathcal{R}_{\mathrm{hf}}(A)),r(M)\}, \label{rhff}%
\end{equation}
for any precompact set in $A.$ Obviously $\overline{F(X)}\subseteq
\mathcal{R}_{\mathrm{hf}}(B(X)))\subseteq\mathcal{R}_{\mathrm{hc}}(B(X))$, so
the inequalities
\begin{equation}
\rho(M/\mathcal{R}_{\mathrm{hc}}(B(X)))\leq\rho(M/\mathcal{R}_{\mathrm{hf}%
}(B(X)))\leq\rho_{f}(M)=\rho_{e}(M) \label{Kineq}%
\end{equation}
are always true for all precompact $M\subset B(X)$; recall that $\rho
_{f}(M)=\rho_{e}(M)$ by Theorem \ref{morr2}.

The inequality $\rho(M/\mathcal{R}_{\mathrm{hc}}(B(X)))\leq\rho_{f}(M)$ in
(\ref{Kineq}) can be strict. For example, if $X$ is an Argyros-Haydon space
then $\rho(M/\mathcal{R}_{\mathrm{hc}}(B(X)))=0$ for each precompact set
$M\subseteq B(X)$ while $\rho_{f}(M)$ can be non-zero by virtue of
semisimplicity $K\left(  X\right)  $). This shows that \textit{even in the
operator case the joint spectral radius formula }(\ref{aGBWF})\textit{ is
stronger than the generalized BW-formula} (\ref{GBWF}).

In general, the inequality $\rho(M/\mathcal{R}_{\mathrm{hc}}(A))\leq
\rho(M/\mathcal{R}_{\mathrm{hf}}(A))$ can be also strict. To see this, let $V$
be the radical compact Banach algebra $\ell_{1}\left(  w\right)  $ considered
in Proposition \ref{radNOTcomp}. As we saw, $V$ has no non-zero nilpotent
elements. Therefore the only finite-rank element $v$ in $V$ is $0$. Indeed,
the multiplication operator L$_{v}$R$_{v}=$ L$_{v^{2}}$ is quasinilpotent. So
if it has finite rank then it is nilpotent:
\[
\text{L}_{v^{2}}^{m}=0.
\]
Applying the operator L$_{v^{2}}^{m}$ to $v$ we have that $v^{2m+1}=0$, i.e.,
$v$ is nilpotent, whence $v=0$.

Let now $A$ be the unitization of $V$. Since all finite-rank elements of $A$
must lie in $V$, it follows from the above that
\[
\mathcal{R}_{\mathrm{hf}}(V)=\mathcal{R}_{\mathrm{hf}}(A)=\left(  0\right)  .
\]
On the other hand, $A$ is hypocompact, whence $A=\mathcal{R}_{\mathrm{hc}}%
(A)$. For $M=\left\{  1\right\}  $ we have that
\[
\rho(M/\mathcal{R}_{\mathrm{hc}}(A))=0\neq1=\rho(M/\mathcal{R}_{\mathrm{hf}%
}(A)).
\]

As usual, in the class of all C*-algebras the situation is simpler.

\begin{theorem}
If $A$ is a C*-algebra then $\mathcal{R}_{\mathrm{hf}}(A)=\mathcal{R}%
_{\mathrm{hc}}(A)$.
\end{theorem}

\begin{proof}
Indeed, if an element $a$ of $A$ is compact then $a^{\ast}a$ is compact and
therefore its spectral projections are finite-rank elements and therefore
belong to $\mathcal{R}_{\text{hf}}(A)$. Since $a^{\ast}a$ is a limit of linear
combinations of its spectral projections we have that
\[
a^{\ast}a\in\mathcal{R}_{\text{hf}}(A).
\]
But it is known (see for example \cite[Proposition 1.4.5]{Ped}) that the
closed ideal generated by $a^{\ast}a$ contains $a$. Thus $\mathcal{R}%
_{\text{hf}}(A)$ contains all compact elements of $A$. Now let $(J_{\alpha
})_{\alpha\leq\gamma}$ be an increasing transfinite chain of closed ideals
with bicompact quotients, and $J_{\gamma}=\mathcal{R}_{\text{hc}}(A)$. Assume
by induction that $J_{\alpha}$ are contained in $\mathcal{R}_{\text{hf}}(A)$
for all $\alpha<\lambda$. If the ordinal $\lambda$ is limit then clearly
$J_{\lambda}$ is also contained in $\mathcal{R}_{\mathrm{hf}}(A)$. Otherwise,
we have that $\lambda=\beta+1$ for some $\beta$. If $a\in J_{\lambda}$ then
$a/J_{\beta}$ is a compact element of $A/J_{\beta}$ whence $a/J_{\beta}%
\in\mathcal{R}_{\mathrm{hf}}(A/J_{\beta})$ and $a\in\mathcal{R}_{\mathrm{hf}%
}(A)$. Therefore, by induction, $\mathcal{R}_{\mathrm{hc}}(A)=J_{\gamma
}=\mathcal{R}_{\mathrm{hf}}(A)$.
\end{proof}

The class of hypocompact C$^{\ast}$-algebras is contained in the class of all
GCR algebras (algebras of type I) and this inclusion is strict: it suffices to
note that even the algebra $C([0,1])$ is not hypocompact. Moreover, there is
an analogue of (\ref{GBWF}) that holds for all C$^{\ast}$-algebras $A$
satisfying some natural restrictions on the space $\text{Prim}(A)$ of all
primitive ideals of $A$:
\begin{equation}
\rho(M)=\max\{\rho(M/\mathcal{R}_{\mathrm{gcr}}(A)),r(M)\}, \label{GBWFCstar}%
\end{equation}
where $\mathcal{R}_{\mathrm{gcr}}(A)$ is the largest GCR ideal of $A$. The map
$A\longmapsto\mathcal{R}_{\mathrm{gcr}}(A)$ is a hereditary topological
radical on the class of all C*-algebras. It follows from (\ref{GBWFCstar})
that any GCR-algebra is a Berger-Wang algebra. The proof and more information
can be found in \cite[Section 10]{ShT14}.

Apart from (\ref{aGBWF}), another version of the joint spectral radius formula
was established in \cite{ShT08}:
\begin{equation}
\rho(M)=\max\{\rho^{\chi}(M),r(M)\} \label{rxf}%
\end{equation}
holds for every precompact set $M$ in $A$, where $\rho^{\chi}(M)$ is defined
as $\rho_{\chi}(\mathrm{L}_{M}\mathrm{R}_{M})^{1/2}$. Unlike $\rho
(M/\mathcal{R}_{\mathrm{hc}}(A))$ and $\rho(M/\mathcal{R}_{\mathrm{hf}}(A)),$
the value $\rho_{\chi}(M)$ is not of the form $\rho(M/J)$, but it deserves
some\textbf{ } interest because it is natural to regard $\rho^{\chi}(M)$ as a
Banach algebraic analogue of $\rho_{\chi}(M)$. By Theorem \ref{morr} and
\cite[Lemma 4.7]{ShT12},
\begin{equation}
\rho^{\chi}(M)=\rho_{\chi}(\mathrm{L}_{M}\mathrm{R}_{M})^{1/2}=\rho
_{e}(\mathrm{L}_{M}\mathrm{R}_{M})^{1/2}\leq\rho\left(  M/J\right)  ^{1/2}%
\rho\left(  M\right)  ^{1/2} \label{x}%
\end{equation}
for every precompact set $M$ in $A$ and every bicompact ideal $J$ of $A$.

In general, $\rho^{\chi}(M)\neq\rho(M/\mathcal{R}_{\mathrm{hc}}(A))$. Indeed,
if $X$ is an Argyros-Haydon space \cite{ArH11} then $B(X)$ is a
one-dimensional extension of $K(X)$. So the algebra $B(X)$ is hypocompact and
$\rho(M/\mathcal{R}_{\mathrm{hc}}(B(X)))=0$ for each precompact set
$M\subseteq B(X)$. On the other hand, for $M=\{1\}$, we see that
$\mathrm{L}_{M}\mathrm{R}_{M}$ is the identity operator on the
infinite-dimensional space $B(X)$, whence $\rho_{\chi}(\mathrm{L}%
_{M}\mathrm{R}_{M})=1$ and $\rho^{\chi}(M)=1$.

\subsection{$BW$-radicals}

In line\textbf{ }with\textbf{ }the above discussion we are looking for such
radicals $\mathcal{P}$ that $\mathcal{P}(A)$ is a $BW$-ideal for each $A$; it
is natural to call them $BW$-\textit{radicals}. Clearly, we are interested in
\textquotedblleft large\textquotedblright\ $BW$-radicals, so that we have to
compare them.

The order for ideal maps, in particular for topological radicals, is
introduced in the usual way: $\mathcal{P}$ $\leq$ $\mathcal{R}$ means that
$\mathcal{P}\left(  A\right)  \subseteq$ $\mathcal{R}\left(  A\right)  $ for
every algebra $A$. For instance, it is obvious that
\[
\mathcal{R}_{\text{hf}}\text{ }\mathcal{\leq R}_{\text{hc}}\text{ and
}\mathcal{R}_{\text{cq}}\leq\text{ Rad},
\]
where $\mathrm{Rad}$ is the \textit{Jacobson radical} $A\longmapsto
\mathrm{Rad}(A)$ (recall that for a Banach algebra $A$, $\mathrm{Rad}(A)$ can
be defined as the largest ideal of $A$ consisting of quasinilpotents). It is
well known that $\mathrm{Rad}$ is hereditary.

As usual, we write $\mathcal{P}$ $<$ $\mathcal{R}$ if $\mathcal{P}$ $\leq$
$\mathcal{R}$ and there is an algebra $A$ such that $\mathcal{P}\left(
A\right)  \neq$ $\mathcal{R}\left(  A\right)  $. For example,
\[
\mathcal{R}_{\text{hf}}<\mathcal{R}_{\text{hc}}<\mathcal{R}_{\text{sc}}\text{
and Rad }<\mathcal{R}_{\text{sc}}.
\]

It is known that, for any family $\mathcal{F}$ of topological radicals, there
exists the smallest upper bound $\vee\mathcal{F}$ and the largest lower bound
$\wedge\mathcal{F}$ of $\mathcal{F}$ in the class of all topological radicals;
clearly, $\vee\mathcal{F}$ and $\wedge\mathcal{F}$ need not belong to
$\mathcal{F}$ itself. If $\mathcal{F}=\left\{  \mathcal{P},\mathcal{R}%
\right\}  $, we write $\mathcal{P}\vee\mathcal{R}$ for $\vee\mathcal{F}$ and
$\mathcal{P}\wedge\mathcal{R}$ for $\wedge\mathcal{F}$. We will describe later
a constructive way for obtaining the radicals $\vee\mathcal{F}$ and
$\wedge\mathcal{F}$.

The following theorem establishes that there is the largest $BW$-radical.

\begin{theorem}
\label{bw}\cite[Theorem 5.9]{ShT12} Let $\mathcal{F}$ be the family of all
$BW$-radicals and $\mathcal{R}_{\mathrm{bw}}=\vee\mathcal{F}$. Then
$\mathcal{R}_{\mathrm{bw}}$ is a $BW$-radical\emph{;} any topological radical
$\mathcal{P}\leq\mathcal{R}_{\mathrm{bw}}$ is a $BW$-radical.
\end{theorem}

The proof uses the structure of radical ideals in\textbf{ }$\vee\mathcal{F}$,
and transfinite stability of the class of $BW$-ideals (see Proposition
\ref{ESBW}).

To show the utility of $\mathcal{R}_{\mathrm{bw}}$, consider the following
example. It follows from Theorems \ref{hcf} and \ref{iness} that
$\mathcal{R}_{\mathrm{hc}}$ and $\mathcal{R}_{\mathrm{cq}}$ are $BW$-radicals.
So, for any Banach algebra $A$, $\mathcal{R}_{\mathrm{hc}}\left(  A\right)  $
and $\mathcal{R}_{\mathrm{cq}}\left(  A\right)  $ are $BW$-ideals. They can
differ; moreover, it can be deduced from Proposition \ref{radNOTcomp} that
there is a Banach algebra $A$ such that $\mathcal{R}_{\mathrm{hc}}\left(
A\right)  $ and $\mathcal{R}_{\mathrm{cq}}\left(  A\right)  $ are both
non-zero, but have zero intersection. The existence of $\mathcal{R}%
_{\mathrm{bw}}$ implies that $\overline{\mathcal{R}_{\mathrm{hc}}\left(
A\right)  +\mathcal{R}_{\mathrm{cq}}\left(  A\right)  }$ is a $BW$-ideal,
because both summands are contained in $\mathcal{R}_{\mathrm{bw}}(A)$. Now one
can further extend this BW-ideal by building an increasing transfinite chain
$\left(  J_{\alpha}\right)  $ of closed ideals such that

\begin{itemize}
\item $J_{0}=\left(  0\right)  $ and $J_{\alpha+1}/J_{\alpha}=$ $\overline
{\mathcal{R}_{\mathrm{hc}}\left(  A/J_{\alpha}\right)  +\mathcal{R}%
_{\mathrm{cq}}\left(  A/_{\alpha}\right)  }$ for all $\alpha$.
\end{itemize}

\noindent In the correspondence with Proposition \ref{ESBW} we conclude that
all $J_{\alpha}$ are $BW$-ideals. It is obvious that there is an ordinal
$\gamma$ such that $J_{\gamma+1}=J_{\gamma}$. It turns out\textbf{ }that
$J_{\gamma}=\left(  \mathcal{R}_{\mathrm{hc}}\vee\mathcal{R}_{\mathrm{cq}%
}\right)  \left(  A\right)  $. To see it and much more, we consider the
details of a construction of radicals\textbf{ }$\vee\mathcal{F}$ and
$\wedge\mathcal{F}$ in the following subsection. Of course, we have that
$\mathcal{R}_{\mathrm{hc}}\vee\mathcal{R}_{\mathrm{cq}}\leq\mathcal{R}%
_{\mathrm{bw}}$, so that the formula%
\begin{equation}
\rho\left(  M\right)  =\max\left\{  \rho\left(  M/\left(  \mathcal{R}%
_{\mathrm{hc}}\vee\mathcal{R}_{\mathrm{cq}}\right)  \left(  A\right)  \right)
,r\left(  M\right)  \right\}  \label{rm}%
\end{equation}
is valid, for any precompact set $M$ in $A$.

It seems that in the Banach algebra context the best candidate for the joint
spectral radius formula is {\large
\begin{equation}
\rho(M)=\max\{\rho(M/\mathcal{R}_{\mathrm{bw}}(A)),r(M)\}. \label{rbwf}%
\end{equation}
} But a priori there can exist a Banach algebra $A$ with non-trivial
$BW$-ideals and with $\mathcal{R}_{\mathrm{bw}}(A))=0$ --- the disadvantage of
formula (\ref{rbwf}) is that the largest $BW$-radical is defined not directly,
since the family of $BW$-radicals is not completely described. However, in
radical context the formula (\ref{rbwf}) is certainly optimal. In particular,
it is stronger\textbf{ } than formula (\ref{aGBWF}) because the largest
$BW$-radical contains the hypocompact radical for any Banach algebra, and the
inclusion can be strict as the above example shows.

In what follows we gather some facts for the better understanding of the
nature of the radical $\mathcal{R}_{\mathrm{bw}}$.

\subsection{Procedures and operations}

Here we describe some ways to construct radicals from preradicals that only
partially satisfy the axioms.

\textit{Procedures} are mappings from one class of ideal maps to another class
of ideal maps. The important examples are the following. If $\mathcal{P}$ and
$\mathcal{R}$ are topological preradicals satisfying $\left(  \mathrm{I1}%
\right)  $ and $\left(  \mathrm{I2}\right)  $, for any algebra $A$, let
$\left(  I_{\alpha}\right)  _{\alpha\leq\gamma}$ and $\left(  J_{\alpha
}\right)  _{\alpha\leq\delta}$ be transfinite chains such that
\begin{equation}
J_{\alpha}=A,\;J_{\alpha+1}=\mathcal{P}\left(  J_{\alpha}\right)  ;\text{
\ }I_{0}=\left(  0\right)  ,\;I_{\alpha+1}=q_{I_{\alpha}}^{-1}\left(
\mathcal{R}\left(  A/I_{\alpha}\right)  \right)  , \label{ja}%
\end{equation}
where $q_{I_{\alpha}}^{{}}:$ $A\longrightarrow A/I_{\alpha}$ is the standard
quotient map. \textbf{ }Then the maps $\mathcal{P}_{\left(  \alpha\right)
^{\circ}}$: $A\longmapsto J_{\alpha}$ and $\mathcal{R}_{\left(  \alpha\right)
^{\ast}}$: $A\longmapsto I_{\alpha}$ are topological preradicals satisfying
$\left(  \mathrm{I1}\right)  $ and $\left(  \mathrm{I2}\right)  $. So
$\mathcal{P}\longmapsto\mathcal{P}_{\left(  \alpha\right)  ^{\circ}}$ and
$\mathcal{R}\longmapsto\mathcal{R}_{\left(  \alpha\right)  ^{\ast}}$ are
procedures ($\alpha$-\textit{superposition} and $\alpha$-\textit{convoluton
procedures}\texttt{)}. The transfinite chains of ideals in (\ref{ja})
stabilize at some steps $\gamma=\gamma\left(  A\right)  $ and $\delta
=\delta\left(  A\right)  $, that is,
\[
I_{\gamma}=I_{\gamma+1}\text{ \ and \ }J_{\delta+1}=J_{\delta}\text{.}%
\]
Set $\mathcal{P}^{\circ}$: $A\longmapsto J_{\delta}$ and $\mathcal{R}^{\ast}$:
$A\longmapsto I_{\gamma}$. Then $\mathcal{P}\longmapsto\mathcal{P}^{\circ}$
and $\mathcal{R}\longmapsto\mathcal{R}^{\ast}$ are called
\textit{superposition} and \textit{convolution} \textit{procedures},\textit{
}respectively;\textit{ }$\mathcal{P}^{\circ}$ \textit{satisfies} $\left(
\mathrm{I1}\right)  $ and $\mathcal{R}^{\ast}$ \textit{satisfies} $\left(
\mathrm{H2}\right)  $ (see \cite[Theorems 6.6 and 6.10]{D97}).

The following two ways of getting new ideal maps are very useful in the
theory. If $\mathcal{F}$ is a family of UTRs then
\[
\mathtt{H}_{\mathcal{F}}\text{: }A\longmapsto\mathtt{H}_{\mathcal{F}}\left(
A\right)  :=\overline{\sum_{\mathcal{R}\in\mathcal{F}}\mathcal{R}\left(
A\right)  }%
\]
is a UTR; if $\mathcal{F}$ consists of OTRs then
\[
\mathtt{B}_{\mathcal{F}}\text{: }A\longmapsto\mathtt{B}_{\mathcal{F}}\left(
A\right)  :=\bigcap\limits_{\mathcal{R}\in\mathcal{F}}\mathcal{R}\left(
A\right)
\]
is an OTR (see \cite[Theorem 4.1]{ShT14}).

Now we extend the action of operations $\vee$ and $\wedge$ introduced in the
preceding subsection. Let $\mathcal{F}$ be a family of topological preradicals
satisfying $\left(  \mathrm{I1}\right)  $ and $\left(  \mathrm{I2}\right)  $.
Set
\begin{equation}
\vee\mathcal{F}=\left(  \mathtt{H}_{\mathcal{F}}\right)  ^{\ast}\text{ and
}\wedge\mathcal{F}=\left(  \mathtt{B}_{\mathcal{F}}\right)  ^{\circ}.
\label{fh}%
\end{equation}
Then $\vee\mathcal{F}$ is the smallest OTR larger than or equal to each
$\mathcal{P}\in\mathcal{F}$; and $\wedge\mathcal{F}$ is the largest UTR
smaller than or equal to each $\mathcal{P}\in\mathcal{F}$. In particular, if
$\mathcal{F}$ consists of UTRs then $\vee\mathcal{F}$ is the smallest
topological radical that is no less than each $\mathcal{P}\in\mathcal{F}$; if
$\mathcal{F}$ consists of OTRs then $\wedge\mathcal{F}$ is the largest
topological radical that does not exceed each $\mathcal{P}\in\mathcal{F}$ (see
\cite[Remark 4.2 and Corollary 4.3]{ShT14}).

\begin{theorem}
\cite[Theorem 8.15]{ShT14} \label{sc} $\mathrm{Rad}\vee\mathcal{R}%
_{\mathrm{hc}}=\mathcal{R}_{\mathrm{sc}}.$
\end{theorem}

If a family $\mathcal{F}$ consists of hereditary topological radicals then%
\[
\wedge\mathcal{F}=\mathtt{B}_{\mathcal{F}}%
\]
is the largest hereditary topological radical that does not exceed each
$P\in\mathcal{F}$ (see \cite[Lemma 3.2]{ShT12}).

As $\mathcal{R}_{\mathrm{hc}}$ and $\mathrm{Rad}$ are hereditary topological
radicals then it follows from Theorem \ref{inf1} that $\mathtt{B}_{\left\{
\mathcal{R}_{\mathrm{hc}},\mathrm{Rad}\right\}  }$ is a hereditary topological
radical $\mathcal{R}_{\mathrm{hc}}\wedge\mathrm{Rad}$ and $\mathcal{R}%
_{\mathrm{hc}}\wedge\mathrm{Rad}\leq\mathcal{R}_{\mathrm{cq}}$. It follows
from Proposition \ref{radNOTcomp} that
\begin{equation}
\mathtt{B}_{\left\{  \mathcal{R}_{\mathrm{hc}},\mathrm{Rad}\right\}
}=\mathcal{R}_{\mathrm{hc}}\wedge\mathrm{Rad}<\mathcal{R}_{\mathrm{cq}}.
\label{hcr}%
\end{equation}

\subsection{Convolution and superposition operations}

In this subsection we prove two useful\textbf{ }lemmas.

For an ideal map $\mathcal{P}$ and a closed ideal $I$ of a Banach algebra $A,$
it is convenient to define an ideal $\mathcal{P}\ast I$ of $A$ by setting
\[
\mathcal{P}\ast I=q_{_{I}}^{-1}\left(  \mathcal{P}\left(  A/I\right)  \right)
\]
where $q_{I}^{{}}$: $A\longrightarrow A/I$ is the standard quotient map.
Clearly, $I\subseteq\mathcal{P}\ast I$. If $\mathcal{P}$ and $\mathcal{R}$ are
topological preradicals satisfying $\left(  \mathrm{I1}\right)  $ and $\left(
\mathrm{I2}\right)  $, define the \textit{convolution} $\mathcal{P}%
\ast\mathcal{R}$ and \textit{superposition} $\mathcal{P}\circ\mathcal{R}$ by
\begin{equation}
\mathcal{P}\ast\mathcal{R}\left(  A\right)  =q_{\mathcal{R}}^{-1}\left(
\mathcal{P}\left(  A/\mathcal{R}\left(  A\right)  \right)  \right)  \text{ and
}\mathcal{P}\circ\mathcal{R}\left(  A\right)  =\mathcal{P}\left(
\mathcal{R}\left(  A\right)  \right)  \label{pr}%
\end{equation}
for every algebra $A,$ where $q_{\mathcal{R}}^{{}}$: $A\longrightarrow
A/\mathcal{R}\left(  A\right)  $ is the standard quotient map. Then
$\mathcal{P}\ast\mathcal{R}$ and $\mathcal{P}\circ\mathcal{R}$ are topological
preradicals satisfying $\left(  \mathrm{I1}\right)  $ and $\left(
\mathrm{I2}\right)  $ (see \cite[Subsection 4.2]{ShT14}); the convolution
operation for preradicals is associative (see \cite[Lemma 4.10]{ShT14}). If
$\mathcal{P}$ and $\mathcal{R}$ are UTRs then so is $\mathcal{P}%
\ast\mathcal{R}$; if $\mathcal{P}$ and $\mathcal{R}$ are OTRs then so is
$\mathcal{P}\circ\mathcal{R}$ (see \cite[Corollary 4.11]{ShT14}).

We underline that one may define the convolution $\mathcal{P}\ast\mathcal{R}$
as above if $\mathcal{P}$ is an ideal map and $\mathcal{R}$ is a closed ideal map.

\begin{lemma}
\label{ifp} If $\mathcal{P}$ is a preradical\emph{,} $\mathcal{R}$ and
$\mathcal{S}$ are closed ideal maps and $\mathcal{R}\leq$ $\mathcal{S},$ then
$\mathcal{P}\ast\mathcal{R}\leq\mathcal{P}\ast$ $\mathcal{S}$ and
$\mathtt{H}_{\left\{  \mathcal{P},\mathcal{R}\right\}  }\leq\mathcal{P}\ast$
$\mathcal{S}$.
\end{lemma}

\begin{proof}
Let $A$ be a Banach algebra, $J=\mathcal{R}\left(  A\right)  $ and
$I=\mathcal{S}\left(  A\right)  $. Let $q_{J}^{{}}:$ $A\longrightarrow A/J$,
$q_{I}^{{}}:$ $A\longrightarrow A/I$ and $q:$ $A/J\longrightarrow A/I$ be the
standard quotient maps. Then $q\circ q_{J}^{{}}=q_{I}^{{}}$ and $q\left(
\mathcal{P}\left(  A/J\right)  \right)  \subseteq\left(  \mathcal{P}\left(
A/I\right)  \right)  $. Therefore
\[
\mathcal{P}\ast\mathcal{R}\left(  A\right)  =q_{J}^{-1}\left(  \mathcal{P}%
\left(  A/J\right)  \right)  \subseteq q_{J}^{-1}q^{-1}q\left(  \mathcal{P}%
\left(  A/J\right)  \right)  \subseteq q_{I}^{-1}\left(  \mathcal{P}\left(
A/I\right)  \right)  =\mathcal{P}\ast\mathcal{S}\left(  A\right)  .
\]
Hence $\mathcal{P}\ast\mathcal{R}\leq\mathcal{P}\ast$ $\mathcal{S}$.

Further, $\mathcal{R}\left(  A\right)  =J\subseteq I$ and $q_{I}^{{}}\left(
\mathcal{P}\left(  A\right)  \right)  \subseteq\mathcal{P}\left(  A/I\right)
$ whence $\mathcal{P}\left(  A\right)  \subseteq q_{I}^{-1}\left(
\mathcal{P}\left(  A/I\right)  \right)  $. Hence%
\[
\mathtt{H}_{\left\{  \mathcal{P},\mathcal{R}\right\}  }\left(  A\right)
=\overline{\mathcal{P}\left(  A\right)  +\mathcal{R}\left(  A\right)
}\subseteq q_{I}^{-1}\left(  \mathcal{P}\left(  A/I\right)  \right)
=\mathcal{P}\ast\mathcal{S}\left(  A\right)  ,
\]
i.e., $\mathtt{H}_{\left\{  \mathcal{P},\mathcal{R}\right\}  }\leq
\mathcal{P}\ast\mathcal{S}$.
\end{proof}

The implication $\mathcal{P}\leq\mathcal{S}\Longrightarrow\mathcal{P}%
\ast\mathcal{R}\leq\mathcal{S}\ast\mathcal{R}$ is obvious.

\begin{lemma}
\label{ip} If $\mathcal{P}$ and $\mathcal{R}$ are UTRs then the radical
$\mathcal{P}\vee\mathcal{R}$ is equal to $\left(  \mathcal{P}\ast
\mathcal{R}\right)  ^{\ast}$\emph{;} if $\mathcal{P}$ and $\mathcal{R}$ are
OTRs then the radical $\mathcal{P}\wedge\mathcal{R}$ is equal to $\left(
\mathcal{P}\circ\mathcal{R}\right)  ^{\circ}$.
\end{lemma}

\begin{proof}
Let $\mathcal{P}$ and $\mathcal{R}$ be UTRs. By Lemma \ref{ifp},
$\mathrm{H}_{\left\{  \mathcal{P},\mathcal{R}\right\}  }\leq\mathcal{P}%
\ast\mathcal{R}\leq$ $\left(  \mathcal{P}\ast\mathcal{R}\right)  ^{\ast}$
whence
\[
\mathcal{P}\vee\mathcal{R}=\left(  \mathtt{H}_{\left\{  \mathcal{P}%
,\mathcal{R}\right\}  }\right)  ^{\ast}\leq\left(  \mathcal{P}\ast
\mathcal{R}\right)  ^{\ast\ast}=\left(  \mathcal{P}\ast\mathcal{R}\right)
^{\ast}.
\]
On the other hand, $\mathcal{P}\ast\mathcal{R}\leq\mathtt{H}_{\left\{
\mathcal{P},\mathcal{R}\right\}  }\ast\mathcal{R}\leq\mathtt{H}_{\left\{
\mathcal{P},\mathcal{R}\right\}  }\ast\mathtt{H}_{\left\{  \mathcal{P}%
,\mathcal{R}\right\}  }$ by Lemma \ref{ifp}. Therefore
\[
\left(  \mathcal{P}\ast\mathcal{R}\right)  ^{\ast}\leq\left(  \mathtt{H}%
_{\left\{  \mathcal{P},\mathcal{R}\right\}  }\ast\mathtt{H}_{\left\{
\mathcal{P},\mathcal{R}\right\}  }\right)  ^{\ast}=\left(  \left(
\mathtt{H}_{\left\{  \mathcal{P},\mathcal{R}\right\}  }\right)  _{\left(
2\right)  ^{\ast}}\right)  ^{\ast}=\left(  \mathtt{H}_{\left\{  \mathcal{P}%
,\mathcal{R}\right\}  }\right)  ^{\ast}=\mathcal{P}\vee\mathcal{R}\text{.}%
\]

Let $\mathcal{P}$ and $\mathcal{R}$ be OTRs. Then $\left(  \mathcal{P}%
\circ\mathcal{R}\right)  ^{\circ}\leq\mathcal{P}\circ\mathcal{R}\leq
\mathtt{B}_{\left\{  \mathcal{P},\mathcal{R}\right\}  }$ whence
\[
\left(  \mathcal{P}\circ\mathcal{R}\right)  ^{\circ}\leq\left(  \mathtt{B}%
_{\left\{  \mathcal{P},\mathcal{R}\right\}  }\right)  ^{\circ}=\mathcal{P}%
\wedge\mathcal{R}\text{.}%
\]
On the other hand, $\mathtt{B}_{\left\{  \mathcal{P},\mathcal{R}\right\}
}\circ\mathtt{B}_{\left\{  \mathcal{P},\mathcal{R}\right\}  }\leq
\mathcal{P}\circ\mathtt{B}_{\left\{  \mathcal{P},\mathcal{R}\right\}  }%
\leq\mathcal{P}\circ\mathcal{R}$. Therefore%
\[
\mathcal{P}\wedge\mathcal{R}=\left(  \mathtt{B}_{\left\{  \mathcal{P}%
,\mathcal{R}\right\}  }\right)  ^{\circ}=\left(  \mathtt{B}_{\left\{
\mathcal{P},\mathcal{R}\right\}  }\circ\mathtt{B}_{\left\{  \mathcal{P}%
,\mathcal{R}\right\}  }\right)  ^{\circ}\leq\left(  \mathcal{P}\circ
\mathcal{R}\right)  ^{\circ}.
\]

\end{proof}

\subsection{Scattered $BW$-radical}

Here we will show that the restriction of $\mathcal{R}_{\mathrm{bw}}$ to the
class of scattered algebras is closely related to radicals of somewhat less
mysterious nature. Namely it coincides with the topological radical
$\mathcal{R}_{\mathrm{hc}}\vee\mathcal{R}_{\mathrm{cq}}$ constructed earlier.

\begin{theorem}
\label{scbw} Let $A$ be a scattered Banach algebra. Then $\mathcal{R}%
_{\mathrm{bw}}\left(  A\right)  =\left(  \mathcal{R}_{\mathrm{hc}}%
\vee\mathcal{R}_{\mathrm{cq}}\right)  \left(  A\right)  $.
\end{theorem}

\begin{proof}
Clearly,\textbf{ } $\mathcal{R}_{\mathrm{hc}}\vee\mathcal{R}_{\mathrm{cq}}%
\leq\mathcal{R}_{\mathrm{bw}}$. Let $I=\left(  \mathcal{R}_{\mathrm{hc}}%
\vee\mathcal{R}_{\mathrm{cq}}\right)  \left(  A\right)  $, $J=\mathcal{R}%
_{\mathrm{bw}}\left(  A\right)  $, $B=A/I$ and $K=J/I$. Then $B$ is a
scattered, $\mathcal{R}_{\mathrm{hc}}\vee\mathcal{R}_{\mathrm{cq}}$-semisimple
algebra and $K\subseteq\mathcal{R}_{\mathrm{bw}}\left(  B\right)  $ is a
closed ideal of $B$. Assume to the contrary that $K\neq\left(  0\right)  $.

As $K$ is a BW-ideal, it is a Berger-Wang algebra. So
\[
\mathrm{Rad}\left(  K\right)  =\mathcal{R}_{\mathrm{cq}}(K)=K\cap
\mathcal{R}_{\mathrm{cq}}(B)\text{ }(\text{we used heredity of }%
\mathcal{R}_{\mathrm{cq}}).
\]
But $\mathcal{R}_{\mathrm{cq}}(B)$ $\subseteq\left(  \mathcal{R}_{\mathrm{hc}%
}\vee\mathcal{R}_{\mathrm{cq}}\right)  \left(  B\right)  =\left(  0\right)  $.
Therefore $\mathrm{Rad}(K)=\left(  0\right)  $, whence $K$ is a semisimple algebra.

Since $B$ is scattered, $K$ is also scattered. By Barnes' Theorem \cite{Ba68}
(see also \cite[Theorem 5.7.8]{Au91} with another proof) $K$ has a non-zero
socle. Since the socle is generated by finite-rank projections, it is a
hypocompact (even hypofinite) ideal and, therefore, is contained in
$\mathcal{R}_{\mathrm{hc}}(K)$. Since $\mathcal{R}_{\mathrm{hc}}$ is a
hereditary radical then
\[
\left(  0\right)  \neq\mathcal{R}_{\mathrm{hc}}\left(  K\right)
=K\cap\mathcal{R}_{\mathrm{hc}}\left(  B\right)  .
\]
But $\mathcal{R}_{\mathrm{hc}}\left(  B\right)  =\left(  0\right)  $, a
contradiction. Hence $K=\left(  0\right)  $, i.e., $J=I$.
\end{proof}

\begin{theorem}
\label{herh} The radical $\mathcal{R}_{\mathrm{hc}}\vee\mathcal{R}%
_{\mathrm{cq}}$ is hereditary.
\end{theorem}

\begin{proof}
Set $\mathcal{P}=\mathcal{R}_{\mathrm{hc}}\vee\mathcal{R}_{\mathrm{cq}}$. Let
$A$ be a Banach algebra and $I$ its closed ideal. If $\mathcal{P}\left(
I\right)  \neq I\cap\mathcal{P}\left(  A\right)  $, let $B=A/\mathcal{P}%
\left(  I\right)  $, $J=I/\mathcal{P}\left(  I\right)  $ and $K=\mathcal{P}%
\left(  A\right)  /\mathcal{P}\left(  I\right)  $. Then $J$ is a $\mathcal{P}%
$-semisimple ideal of $B$. Therefore
\begin{equation}
\mathcal{R}_{\mathrm{hc}}\left(  J\right)  =\mathcal{R}_{\mathrm{cq}}\left(
J\right)  =\left(  0\right)  \label{semh}%
\end{equation}
and $K$ is an ideal of $B$ contained in $\mathcal{P}\left(  B\right)  $. Hence
$K\in BW\left(  B\right)  $. As $\mathcal{P}\left(  B\right)  \subseteq
\mathcal{R}_{\mathrm{sc}}\left(  B\right)  $ and $\mathcal{R}_{\mathrm{sc}}$
is hereditary, $K$ is a scattered algebra.

Let $L=J\cap K$. Then $L$ is a non-zero ideal of $B$. As $L$ is an ideal of
$K$,\textbf{ } $L\in BW\left(  B\right)  $ and $L$ is scattered. As $L$ is an
ideal of $J,$ it follows from (\ref{semh}) that
\begin{equation}
\mathcal{R}_{\mathrm{hc}}\left(  L\right)  =\mathcal{R}_{\mathrm{cq}}\left(
L\right)  =\left(  0\right)  . \label{semc}%
\end{equation}

As $L$ is a Berger-Wang algebra, it follows from (\ref{bwa}) that
$\operatorname*{Rad}\left(  L\right)  \subseteq\mathcal{R}_{\mathrm{cq}%
}\left(  L\right)  \mathrm{.}$ Therefore $L$ is a semisimple non-zero Banach
algebra by (\ref{semc}). By Barnes' Theorem \cite{Ba68}, $L$ has non-zero
socle $\mathrm{soc}\left(  L\right)  $, i.e., $\mathcal{R}_{\mathrm{hc}%
}\left(  L\right)  \neq\left(  0\right)  $, a contradiction.
\end{proof}

\begin{theorem}
\label{hc} $\mathcal{R}_{\mathrm{hc}}\vee\mathcal{R}_{\mathrm{cq}}%
=\mathcal{R}_{\mathrm{bw}}\circ\mathcal{R}_{\mathrm{sc}}=\mathcal{R}%
_{\mathrm{bw}}\wedge\mathcal{R}_{\mathrm{sc}}$.
\end{theorem}

\begin{proof}
Let $A$ be a Banach algebra and $I=\mathcal{R}_{\mathrm{sc}}\left(  A\right)
$. Then $I$ is a scattered algebra and $\mathcal{R}_{\mathrm{bw}}\left(
I\right)  =\left(  \mathcal{R}_{\mathrm{hc}}\vee\mathcal{R}_{\mathrm{cq}%
}\right)  \left(  I\right)  $ by Theorem \ref{scbw}. As $\mathcal{R}%
_{\mathrm{hc}}\vee\mathcal{R}_{\mathrm{cq}}\leq\mathcal{R}_{\mathrm{sc}}$ and
$\mathcal{R}_{\mathrm{hc}}\vee\mathcal{R}_{\mathrm{cq}}\leq$ $\mathcal{R}%
_{\mathrm{bw}}$ , we obtain that
\begin{align*}
\left(  \mathcal{R}_{\mathrm{hc}}\vee\mathcal{R}_{\mathrm{cq}}\right)  \left(
A\right)   &  =\left(  \mathcal{R}_{\mathrm{hc}}\vee\mathcal{R}_{\mathrm{cq}%
}\right)  \left(  \left(  \mathcal{R}_{\mathrm{hc}}\vee\mathcal{R}%
_{\mathrm{cq}}\right)  \left(  A\right)  \right) \\
&  \subseteq\left(  \mathcal{R}_{\mathrm{hc}}\vee\mathcal{R}_{\mathrm{cq}%
}\right)  \left(  \mathcal{R}_{\mathrm{sc}}\left(  A\right)  \right)
\subseteq\mathcal{R}_{\mathrm{bw}}\left(  \mathcal{R}_{\mathrm{sc}}\left(
A\right)  \right) \\
&  =\mathcal{R}_{\mathrm{bw}}\left(  I\right)  =\left(  \mathcal{R}%
_{\mathrm{hc}}\vee\mathcal{R}_{\mathrm{cq}}\right)  \left(  I\right) \\
&  \subseteq\left(  \mathcal{R}_{\mathrm{hc}}\vee\mathcal{R}_{\mathrm{cq}%
}\right)  \left(  A\right)  ,
\end{align*}
i.e., $\mathcal{R}_{\mathrm{hc}}\vee\mathcal{R}_{\mathrm{cq}}=\mathcal{R}%
_{\mathrm{bw}}\circ\mathcal{R}_{\mathrm{sc}}$.

It is clear that $\mathcal{R}_{\mathrm{bw}}\circ\mathcal{R}_{\mathrm{sc}}%
\leq\mathcal{R}_{\mathrm{bw}}$ and $\mathcal{R}_{\mathrm{bw}}\circ
\mathcal{R}_{\mathrm{sc}}\leq\mathcal{R}_{\mathrm{sc}}$. As $\mathcal{R}%
_{\mathrm{bw}}\circ\mathcal{R}_{\mathrm{sc}}$ is a topological radical, then
$\mathcal{R}_{\mathrm{bw}}\circ\mathcal{R}_{\mathrm{sc}}\leq\mathcal{R}%
_{\mathrm{bw}}\wedge\mathcal{R}_{\mathrm{sc}}$. By Lemma \ref{ip},
\[
\mathcal{R}_{\mathrm{bw}}\wedge\mathcal{R}_{\mathrm{sc}}=\left(
\mathcal{R}_{\mathrm{bw}}\circ\mathcal{R}_{\mathrm{sc}}\right)  ^{\circ}%
\leq\mathcal{R}_{\mathrm{bw}}\circ\mathcal{R}_{\mathrm{sc}}\leq\mathcal{R}%
_{\mathrm{bw}}\wedge\mathcal{R}_{\mathrm{sc}}.
\]
Therefore $\mathcal{R}_{\mathrm{bw}}\circ\mathcal{R}_{\mathrm{sc}}%
=\mathcal{R}_{\mathrm{bw}}\wedge\mathcal{R}_{\mathrm{sc}}$.
\end{proof}

Let us call $\mathcal{R}_{\mathrm{sbw}}:=\mathcal{R}_{\mathrm{sc}}%
\wedge\mathcal{R}_{\mathrm{bw}}$ the \textit{scattered }$BW$\textit{-radical}.

\subsection{ The centralization procedure}

Our next aim is to remove the frame of the class of scattered algebras by
adding commutative algebras and forming transfinite extensions. For this in
the theory of topological radicals there exists a special procedure.

Let ${\sum\nolimits_{a}}\left(  A\right)  $ be the sum of all commutative
ideals of $A$, and let ${\sum\nolimits_{\beta}}\left(  A\right)  $ be the sum
of all nilpotent ideals. The maps ${\sum\nolimits_{a}}$ and ${\sum
\nolimits_{\beta}}$ are preradicals on the class of Banach algebras.

Note that the ideals ${\sum\nolimits_{a}}\left(  A\right)  $ and
${\sum\nolimits_{\beta}}\left(  A\right)  $ can be non-closed.

If $A$ is semiprime then ${\sum\nolimits_{a}}\left(  A\right)  $ is the
largest central ideal of $A$ (see \cite[Lemma 5.1]{ShT14}). Let $\mathcal{P}$
be a closed ideal map on the class of Banach algebras. Define an ideal map
$\mathcal{P}^{a}$ by setting
\[
\mathcal{P}^{a}={\sum\nolimits_{a}}\ast\mathcal{P}.
\]
Let ${\sum\nolimits_{\beta}\leq}$ $\mathcal{P}.$ Then $\mathcal{P}^{a}\left(
A\right)  $ is the largest ideal of $A$ commutative modulo $\mathcal{P}\left(
A\right)  $, and if $\mathcal{P}$ is a topological radical then, by
\cite[Theorem 5.3]{ShT14}, $\mathcal{P}^{a}$ is a\textit{ }UTR.

\begin{proposition}
Let $\mathcal{F}$ be a family of topological radicals\emph{,} let
${\sum\nolimits_{\beta}}\leq\mathcal{P}\in\mathcal{F}$ and $\mathcal{G}%
=\mathcal{F}\backslash\left\{  \mathcal{P}\right\}  $. Then $\left(
\mathtt{H}_{\mathcal{F}}\right)  ^{a}\leq\mathcal{P}^{a}\ast\mathtt{H}%
_{\mathcal{G}}$ and $\left(  \mathtt{H}_{\mathcal{F}}\right)  ^{a\ast}=\left(
\mathcal{P}^{a}\ast\mathtt{H}_{\mathcal{G}}\right)  ^{\ast}$.
\end{proposition}

\begin{proof}
Let $\mathcal{T}=\mathtt{H}_{\mathcal{G}}$. Then $\mathcal{T}$ is a UTR. As
the convolution operation is associative then
\begin{equation}
\mathcal{P}^{a}\ast\mathcal{T}=\left(  {\sum\nolimits_{a}}\ast\mathcal{P}%
\right)  \ast\mathcal{T}={\sum\nolimits_{a}}\ast\left(  \mathcal{P}%
\ast\mathcal{T}\right)  =\left(  \mathcal{P}\ast\mathcal{T}\right)  ^{a}.
\label{pt}%
\end{equation}
By Lemma \ref{ifp}, $\mathtt{H}_{\mathcal{F}}\leq\mathcal{P}\ast\mathcal{T}$.
Then
\begin{equation}
\left(  \mathtt{H}_{\mathcal{F}}\right)  ^{a}\leq\left(  \mathcal{P}%
\ast\mathcal{T}\right)  ^{a}. \label{hf}%
\end{equation}

Let $\mathcal{R}=\left(  \mathtt{H}_{\mathcal{F}}\right)  ^{a\ast}$ and
$\mathcal{S}=\left(  \mathcal{P}\ast\mathcal{T}\right)  ^{a\ast}$. It follows
from (\ref{hf}) that $\left(  \mathtt{H}_{\mathcal{F}}\right)  ^{a}\leq\left(
\mathcal{P}\ast\mathcal{T}\right)  ^{a}\leq\mathcal{S}$ and, therefore,
\[
\mathcal{R}=\left(  \mathtt{H}_{\mathcal{F}}\right)  ^{a\ast}\leq
\mathcal{S}^{\ast}=\mathcal{S}.
\]

On the other hand, $\mathtt{H}_{\mathcal{F}}\leq\mathcal{R}$ whence
$\mathcal{P}\ast\mathcal{T}\leq\mathcal{R}\ast\mathcal{R}=\mathcal{R}$,
$\left(  \mathcal{P}\ast\mathcal{T}\right)  ^{a}\leq\mathcal{R}^{a}%
=\mathcal{R}$ and
\[
\mathcal{S}=\left(  \mathcal{P}\ast\mathcal{T}\right)  ^{a\ast}\leq
\mathcal{R}^{\ast}=\mathcal{R}.
\]

\end{proof}

Sometimes $\mathcal{P}^{a}$ is a topological radical if $\mathcal{P}$ is a
topological radical. We have the following

\begin{theorem}
\cite[Theorem 5.13]{ShT12} $\mathcal{R}_{\mathrm{cq}}^{a}$ is a hereditary
$BW$-radical.
\end{theorem}

\begin{corollary}
\cite[Corollary 5.15]{ShT12} $\mathcal{R}_{\mathrm{hc}}\vee\mathcal{R}%
_{\mathrm{cq}}^{a}$ is a $BW$-radical.
\end{corollary}

\subsection{Centralization of $BW$-radicals and continuity of the joint
spectral radius}

\begin{lemma}
\label{bwha} $\mathcal{R}_{\mathrm{bw}}^{a}\left(  A\right)  $ and
$\mathcal{R}_{\mathrm{sbw}}^{a}\left(  A\right)  $ are $BW$-ideals for every
Banach algebra $A$.
\end{lemma}

\begin{proof}
Indeed, $\rho\left(  M\right)  =\max\left\{  \rho\left(  M/\mathcal{R}%
_{\mathrm{bw}}\left(  A\right)  \right)  ,r\left(  M\right)  \right\}  $ for
every precompact set $M$ in $A$ by definition of $\mathcal{R}_{\mathrm{bw}}$.
Let $B=A/\mathcal{R}_{\mathrm{bw}}\left(  A\right)  $ and $N=M/\mathcal{R}%
_{\mathrm{bw}}\left(  A\right)  $. As $%
{\textstyle\sum\nolimits_{\beta}}
\leq\mathcal{R}_{\mathrm{cq}}\leq\mathcal{R}_{\mathrm{bw}}$, $B$ is semiprime
and $%
{\textstyle\sum\nolimits_{a}}
\left(  B\right)  $ is the largest central ideal of $B.$ It is clear that $%
{\textstyle\sum\nolimits_{a}}
\left(  B\right)  $ is closed. By \cite[Lemma 5.5]{ShT12},%
\[
\rho\left(  N\right)  =\max\left\{  \rho\left(  N/%
{\textstyle\sum\nolimits_{a}}
\left(  B\right)  \right)  ,r_{1}\left(  N\right)  \right\}
\]
where $r_{1}\left(  N\right)  =\sup\left\{  \rho\left(  a\right)  \text{:
}a\in N\right\}  \leq r\left(  N\right)  $. Hence $%
{\textstyle\sum\nolimits_{a}}
\left(  B\right)  =\mathcal{R}_{\mathrm{bw}}^{a}\left(  A\right)
/\mathcal{R}_{\mathrm{bw}}\left(  A\right)  $ and $\mathcal{R}_{\mathrm{bw}%
}\left(  A\right)  $ are $BW$-ideals. By Proposition \ref{ESBW}, $BW$-ideals
are stable with respect to extensions. So $\mathcal{R}_{\mathrm{bw}}%
^{a}\left(  A\right)  $ is a $BW$-ideal.

As $\mathcal{R}_{\mathrm{sbw}}^{a}\left(  A\right)  \subseteq\mathcal{R}%
_{\mathrm{bw}}^{a}\left(  A\right)  ,$ then $\mathcal{R}_{\mathrm{sbw}}%
^{a}\left(  A\right)  $ is also a $BW$-ideal.
\end{proof}

\begin{theorem}
$\mathcal{R}_{\mathrm{sbw}}^{a\ast}$ is a $BW$-radical and $\mathcal{R}%
_{\mathrm{bw}}^{a}=\mathcal{R}_{\mathrm{bw}}$.
\end{theorem}

\begin{proof}
Let $\mathcal{P}$ be $\mathcal{R}_{\mathrm{sbw}}$ or $\mathcal{R}%
_{\mathrm{bw}}$. Clearly, ${\sum\nolimits_{\beta}}\leq\mathcal{P}$. Let $A$ be
a Banach algebra, and let $\left(  J_{\alpha}\right)  _{\alpha\leq\gamma+1}$
be an increasing transfinite chain of closed ideals of $A$ such that
$J_{1}=\mathcal{P}^{a}\left(  A\right)  $ and $J_{\gamma+1}=J_{\gamma}$, and
$J_{\alpha+1}/J_{\alpha}=\mathcal{P}^{a}\left(  A/J_{\alpha}\right)  $ for all
$\alpha\leq\gamma$. By Lemma \ref{bwha}, ideals $J_{1}$ and $J_{\alpha
+1}/J_{\alpha}$ are $BW$-ideals. By Proposition \ref{ESBW}, $\mathcal{P}%
^{a\ast}\left(  A\right)  $ is a $BW$-ideal for every Banach algebra $A$, that
is, $\mathcal{P}^{a\ast}$ is a $BW$-radical.

As $\mathcal{R}_{\mathrm{bw}}$ is the largest $BW$-radical, then
$\mathcal{R}_{\mathrm{bw}}^{a\ast}\leq\mathcal{R}_{\mathrm{bw}}$. We obtain
that
\[
\mathcal{R}_{\mathrm{bw}}\leq\mathcal{R}_{\mathrm{bw}}^{a}\leq\mathcal{R}%
_{\mathrm{bw}}^{a\ast}\leq\mathcal{R}_{\mathrm{bw}}%
\]
whence $\mathcal{R}_{\mathrm{bw}}^{a}=\mathcal{R}_{\mathrm{bw}}$.
\end{proof}

This formally gives the following

\begin{corollary}
Any Banach algebra $A$ commutative modulo the radical $\mathcal{R}%
_{\mathrm{bw}}\left(  A\right)  $ is $\mathcal{R}_{\mathrm{bw}}$-radical.
\end{corollary}

\begin{proof}
Let $B=A/\mathcal{R}_{\mathrm{bw}}\left(  A\right)  $. Clearly $B$
$=\mathcal{R}^{a}\left(  B\right)  $ for every topological radical
$\mathcal{R}$. Therefore
\[
B=\mathcal{R}_{\mathrm{bw}}^{a}\left(  B\right)  =\mathcal{R}_{\mathrm{bw}%
}\left(  B\right)  =\mathcal{R}_{\mathrm{bw}}\left(  A/\mathcal{R}%
_{\mathrm{bw}}\left(  A\right)  \right)  =\left(  0\right)
\]
whence $A=\mathcal{R}_{\mathrm{bw}}\left(  A\right)  $.
\end{proof}

\begin{theorem}
\label{sbw}$\mathcal{R}_{\mathrm{sbw}}^{a\ast}=\mathcal{R}_{\mathrm{hc}}%
\vee\mathcal{R}_{\mathrm{cq}}^{a}$.
\end{theorem}

\begin{proof}
Let $\mathcal{S}=\left(  \left(  \mathcal{R}_{\mathrm{cq}}\ast\mathcal{R}%
_{\mathrm{hc}}\right)  ^{a}\right)  ^{\ast}$. By Lemma \ref{ip} and formula
(\ref{pt}) applied to $\mathcal{P}=\mathcal{R}_{\mathrm{cq}}$ and
$\mathcal{T}=\mathcal{R}_{\mathrm{hc}}$, we have that
\[
\mathcal{R}_{\mathrm{sbw}}=\mathcal{R}_{\mathrm{hc}}\vee\mathcal{R}%
_{\mathrm{cq}}\leq\mathcal{R}_{\mathrm{hc}}\vee\mathcal{R}_{\mathrm{cq}}%
^{a}=\left(  \mathcal{R}_{\mathrm{cq}}^{a}\ast\mathcal{R}_{\mathrm{hc}%
}\right)  ^{\ast}=\left(  \left(  \mathcal{R}_{\mathrm{cq}}\ast\mathcal{R}%
_{\mathrm{hc}}\right)  ^{a}\right)  ^{\ast}=\mathcal{S}.
\]
Let $A$ be a Banach algebra, and let $I=\mathcal{S}\left(  A\right)  $. By
Lemma \ref{ip},
\[
\mathcal{R}_{\mathrm{sbw}}^{a}(A)\subseteq\mathcal{S}^{a}\left(  A\right)
=q_{_{I}}^{-1}\left(  {\textstyle\sum\nolimits_{a}}\left(  A/I\right)
\right)  .
\]
As $\left(  \mathcal{R}_{\mathrm{cq}}\ast\mathcal{R}_{\mathrm{hc}}\right)
^{a}\ast\mathcal{S}=\mathcal{S}$ then $\mathcal{R}_{\mathrm{hc}}\left(
A/I\right)  =\left(  0\right)  $. Indeed, if $\mathcal{R}_{\mathrm{hc}}\left(
A/I\right)  \neq\left(  0\right)  $ then
\[
\mathcal{R}_{\mathrm{hc}}\ast\mathcal{S}\left(  A\right)  =q_{_{I}}%
^{-1}\left(  \mathcal{R}_{\mathrm{hc}}\left(  A/I\right)  \right)
\]
differs from $I$ and
\begin{align*}
\left(  \mathcal{R}_{\mathrm{cq}}\ast\mathcal{R}_{\mathrm{hc}}\right)
^{a}\ast\mathcal{S}\left(  A\right)   &  =\left(  \mathcal{R}_{\mathrm{cq}%
}^{a}\ast\mathcal{R}_{\mathrm{hc}}\right)  \ast\mathcal{S}\left(  A\right)
=\mathcal{R}_{\mathrm{cq}}^{a}\ast\left(  \mathcal{R}_{\mathrm{hc}}%
\ast\mathcal{S}\right)  \left(  A\right) \\
&  \neq I=\mathcal{S}\left(  A\right)  ,
\end{align*}
a contradiction. Therefore $\mathcal{R}_{\mathrm{hc}}\ast\mathcal{S}%
=\mathcal{S}$.

Similarly, we obtain that $\mathcal{R}_{\mathrm{cq}}\left(  A/I\right)
=\left(  0\right)  $ and $\mathcal{R}_{\mathrm{cq}}\ast\mathcal{S}%
=\mathcal{S}$. Then
\begin{align*}
\mathcal{R}_{\mathrm{sbw}}^{a}(A)  &  \subseteq\mathcal{S}^{a}\left(
A\right)  ={\sum\nolimits_{a}}\ast\mathcal{S}\left(  A\right)  ={\sum
\nolimits_{a}}\ast\left(  \mathcal{R}_{\mathrm{cq}}\ast\left(  \mathcal{R}%
_{\mathrm{hc}}\ast\mathcal{S}\right)  \right)  \left(  A\right) \\
&  =\left(  {\sum\nolimits_{a}}\ast\left(  \mathcal{R}_{\mathrm{cq}}%
\ast\mathcal{R}_{\mathrm{hc}}\right)  \right)  \ast\mathcal{S}\left(
A\right)  =\left(  \mathcal{R}_{\mathrm{cq}}\ast\mathcal{R}_{\mathrm{hc}%
}\right)  ^{a}\ast\mathcal{S}\left(  A\right) \\
&  =\mathcal{S}\left(  A\right)  ,
\end{align*}
i.e., $\mathcal{R}_{\mathrm{sbw}}^{a}\leq\mathcal{S}$. As $\mathcal{S}$ is a
topological radical,%
\[
\mathcal{R}_{\mathrm{sbw}}^{a\ast}=\left(  \mathcal{R}_{\mathrm{sbw}}%
^{a}\right)  ^{\ast}\leq\mathcal{S}^{\ast}=\mathcal{S}=\mathcal{R}%
_{\mathrm{hc}}\vee\mathcal{R}_{\mathrm{cq}}^{a}.
\]

On the other hand, as $\mathcal{R}_{\mathrm{cq}}\ast\mathcal{R}_{\mathrm{hc}%
}\leq\left(  \mathcal{R}_{\mathrm{hc}}\vee\mathcal{R}_{\mathrm{cq}}\right)
\ast\left(  \mathcal{R}_{\mathrm{hc}}\vee\mathcal{R}_{\mathrm{cq}}\right)
=\mathcal{R}_{\mathrm{hc}}\vee\mathcal{R}_{\mathrm{cq}}=\mathcal{R}%
_{\mathrm{sbw}},$
\[
\left(  \mathcal{R}_{\mathrm{cq}}\ast\mathcal{R}_{\mathrm{hc}}\right)
^{a}\leq\mathcal{R}_{\mathrm{sbw}}^{a}%
\]
and then%
\[
\mathcal{R}_{\mathrm{hc}}\vee\mathcal{R}_{\mathrm{cq}}^{a}=\mathcal{S}=\left(
\left(  \mathcal{R}_{\mathrm{cq}}\ast\mathcal{R}_{\mathrm{hc}}\right)
^{a}\right)  ^{\ast}\leq\left(  \mathcal{R}_{\mathrm{sbw}}^{a}\right)  ^{\ast
}=\mathcal{R}_{\mathrm{sbw}}^{a\ast}.
\]

\end{proof}

We will\textbf{ }mention now \textbf{ }an application of this result to the
problem of continuity of joint spectral radius. Let us recall the required
definitions. Consider the function $M\longmapsto\rho\left(  M\right)  $ for
bounded sets $M$ of a Banach algebra $A$. This function is upper continuous
(see \cite[Theorem 3.1]{ShT00}), that is,%
\begin{equation}
\lim\sup\rho\left(  M_{n}\right)  \leq\rho\left(  M\right)  \label{upp}%
\end{equation}
when $M_{n}$ converges to $M$ in the Hausdorff metric. The set $M$ is a
\textit{point of continuity of the joint spectral radius} if $\rho\left(
M_{n}\right)  \rightarrow\rho\left(  M\right)  $ for every sequence $\left(
M_{n}\right)  $ convergent to $M$.

\begin{corollary}
\label{con} Let $M$ be a precompact set in a Banach algebra $A$. If
$\rho\left(  M/\mathcal{R}_{\mathrm{sbw}}^{a\ast}\left(  A\right)  \right)
<\rho\left(  M\right)  $ then $M$ is a point of continuity of the joint
spectral radius.
\end{corollary}

\begin{proof}
By\textbf{ }virtue of Theorem \ref{sbw} it is sufficient to remark that $M$ is
a point of continuity of the joint spectral radius if $\rho\left(  M/\left(
\mathcal{R}_{\mathrm{hc}}\vee\mathcal{R}_{\mathrm{cq}}^{a}\right)  \left(
A\right)  \right)  <\rho\left(  M\right)  $ by \cite[Theorem 6.3]{ShT12}.
\end{proof}

The following corollary is a consequence of \cite[Corollary 6.4]{ShT12} and
Theorem \ref{sbw}.

\begin{corollary}
Let $A$ be a Banach algebra\emph{,} and let $G$ be a semigroup in
$\mathcal{R}_{\mathrm{sbw}}^{a\ast}\left(  A\right)  .$ If $G$ consists of
quasinilpotent elements of $A$ then the closed subalgebra $\overline{A\left(
G\right)  }$ generated by $G$ is compactly quasinilpotent.
\end{corollary}

\begin{proof}
As $G$ consists of quasinilpotent elements, $r\left(  M\right)  =0$ for every
precompact set $M$ in $G.$ As $\mathcal{R}_{\mathrm{sbw}}^{a\ast}%
\leq\mathcal{R}_{\mathrm{bw}},$ then $\rho\left(  M\right)  =r\left(
M\right)  $ for every precompact set $M$ in $\mathcal{R}_{\mathrm{sbw}}%
^{a\ast}\left(  A\right)  $. Hence $\rho\left(  M\right)  =0$ for every
precompact set $M$ in $G$. As it was described above (see, for instance,
\cite[Proposition 3.5]{T85}), $A\left(  G\right)  $ is finitely
quasinilpotent. It follows from Corollary \ref{con} that $\rho$ is continuous
at any precompact set in $\mathcal{R}_{\mathrm{sbw}}^{a\ast}\left(  A\right)
$. As the closure $\overline{A\left(  G\right)  }$ is contained in
$\mathcal{R}_{\mathrm{sbw}}^{a\ast}\left(  A\right)  $, and each compact
subset of $\overline{A\left(  G\right)  }$ is a limit of a net of finite
subsets of $A(G)$, the algebra $\overline{A\left(  G\right)  }$ is compactly quasinilpotent.
\end{proof}


\begin{thebibliography}{99999}                                                                                            %


\bibitem[AH11]{ArH11}S. A. Argyros, R. Haydon, A hereditarily indecomposable
$L_{\infty}$-space that solves the scalar-plus-compact problem, Acta Math.,
206 ( 2011), 1--54.

\bibitem[Aup91]{Au91}B. Aupetit, A primer on spectral theory, Springer, New
York, 1991.

\bibitem[Bar68]{Ba68}B. A. Barnes, On the existence of minimal ideals in a
Banach algebra, Trans. AMS 133 (1968), 511-517.

\bibitem[BW92]{BW92}M.A. Berger, Y. Wang, Bounded semigroups of matrices,
Linear Algebra and Appl., 166 (1992), 21--27.

\bibitem[CT16]{CT16}P. Cao, Yu. V. Turovskii, Topological radicals, VI.
Scattered elements in Banach Jordan and associative algebras, Studia Math. 235
(2) (2016), 171-208.

\bibitem[Dix97]{D97}P.G. Dixon, Topologically irreducible representations and
radicals in Banach algebras. Proc. London Math. Soc. (3) 74 (1997), 174--200.

\bibitem[GM65]{GM65}L. S. Goldenstein and A. S. Markus, On a measure of
non-compactness of bounded sets and linear operators, in \textquotedblleft
Studies in Algebra and Mathematic Analysis,\textquotedblright\ pp. 45-54,
Izdat. Karta Moldovenjaski, Kishinev, 1965. [In Russian]

\bibitem[Jun09]{Jun09}R. Jungers, Joint spectral radius, theory and
applcations, Berlin: Springer-Verlag, 2009, 197~p.

\bibitem[Lom73]{L73}V. I. Lomonosov, Invariant subspaces for operators
commuting with compact operators, Funct. Anal. and Appl., 7 (1973), 213--21.

\bibitem[KeST9]{KeST}M. Kennedy, V. S. Shulman, Yu. V. Turovskii, Invariant
subspaces of subgraded Lie algebras of compact operators, Integr. Eq. Oper.
Theory, 63 (2009), 47--93.

\bibitem[KST09]{KShT09}E. Kissin, V. S. Shulman, Yu. V. Turovskii, Banach Lie
algebras with Lie subalgebras of finite codimension have Lie ideals, J. London
Math. Soc., 80 (2009), 603--626.

\bibitem[KST12]{KShT12}E. Kissin, V. S. Shulman, Yu. V. Turovskii, Topological
radicals and Frattini theory of Banach Lie algebras, Integral Equations and
Operator theory, 74 (2012), 51--121.

\bibitem[M09*]{Mor09}I. D. Morris, The generalized Berger-Wang formula and the
spectral radius of linear cocycles, preprint: ArXiv:0906.2915v1 [math.DS] 16
Jun 2009.

\bibitem[Mor12]{Mor12}I. D. Morris, The generalized Berger-Wang formula and
the spectral radius of linear cocycles, J. Funct. Anal. 262 (2012), 811--824.

\bibitem[Ped18]{Ped}G.-K. Pedersen, C*-algebras and their authomorphism
groups, London: Academic Press, 2018

\bibitem[RS60]{RS60}G.-C. Rota, W. G. Strang, A note on the joint spectral
radius, Indag. Math., 22 (1960), 379--381.

\bibitem[Sh84]{Sh84}V. S. Shulman, On invariant subspaces of Volterra
operators, Funct. Anal. Appl., 18 (1984), 84-85. [In Russian]

\bibitem[ST00]{ShT00}V. S. Shulman, Yu. V. Turovskii, Joint spectral radius,
operator semigroups and a problem of Wojtynskii, J. Funct. Anal. 177 (2000), 383--441.

\bibitem[ST01]{ShT01}V. S. Shulman, Yu. V. Turovskii, Radicals in Banach
algebras and some unsolved problems in the theory of radical Banach algebras,
Funct. Anal. and its Appl., 35 (2001), 312-314.

\bibitem[ST02]{ShT02}V. S. Shulman, Yu. V. Turovskii, Formulae for joint
spectral radii of sets of operators, Studia Math., 149 (2002), 23--37.

\bibitem[ST'05]{ShT05}V. S. Shulman, Yu. V. Turovskii, Invariant subspaces of
operator Lie algebras and Lie algebras with compact adjoint action, J. Funct.
Anal., 223 (2005), 425--508.

\bibitem[ST05]{ShT05r}V. S. Shulman, Yu. V. Turovskii, Topological radicals,
I. Basic properties, tensor products and joint quasinilpotence, Banach Center
Publications, 67 (2005), 293--333.

\bibitem[ST08*]{ShT08}V. S. Shulman, Yu. V. Turovskii, Application of
topological radicals to calculation of joint spectral radius, preprint:
arXiv:0805.0209 [math.FA] 2 May 2008.

\bibitem[ST10]{ShT10}V. S. Shulman., Yu. V. Turovskii, Topological radicals,
II. Applications to the spectral theory of multiplication operators, Operator
Theory: Advances and Applications, 212 (2010), 45--114.

\bibitem[ST12]{ShT12}V. S. Shulman, Yu. V. Turovskii, Topological radicals and
joint spectral radius, Funct. Anal. and its Appl., 46 (2012), 287-304.

\bibitem[ST14]{ShT14}V. S. Shulman, Yu. V. Turovskii, Topological radicals, V.
From algebra to spectral theory, Operator Theory Advances and Applications
233, Algebraic Methods in Functional analysis (2014), 171--280.

\bibitem[Th87]{Th87}P. Thieullen, Fibres dynamiques asymptotiquement compacts.
Exposants de Lyapunov. Entropie. Dimension., Ann. Inst. H. Poincare Anal. Non
Lineare 4 (1) (1987) 49-97.

\bibitem[Tur85]{T85}Yu. V. Turovskii, Spectral properties of certain Lie
subalgebras and the spectral radius of subsets of a Banach algebra, in
\textquotedblleft Spectral Theory of Operators and Its
Applications\textquotedblright\ , Vol. 6, 144-181, Elm, Baku, 1985. [In Russian]

\bibitem[Tur99]{T99}Yu. V. Turovskii, Volterra semigroups have invariant
subspaces, J. Funct. Anal. 182 (1999), 313--323.

\bibitem[Val64]{Vala64}K. Vala, On compact sets of compact operators, Ann.
Acad. Sci. Fenn. Ser. A I, 351 (1964), 1--8.
\end{thebibliography}
\end{document}